\documentclass{amsart}
\usepackage{amsmath}
\usepackage{amsfonts}

\theoremstyle{plain}

\newtheorem{corollary}{Corollary}

\newtheorem{lemma}{Lemma}

\newtheorem{proposition}{Proposition}
\newtheorem{remark}{Remark}

\numberwithin{equation}{section}
\input{tcilatex}

\begin{document}
\title[Expansions of densities]{Expansions of one density via polynomials
orthogonal with respect to the other.}
\author{Pawe\l\ J. Szab\l owski}
\address{Department of Mathematics and Information Sciences\\
Warsaw University of Technology\\
pl. Politechniki 1, \\
00-661 Warszawa, Poland}
\email{pawel.szablowski@gmail.com,pszablowski@elka.pw.edu.pl}
\date{10 February 2009}

\begin{abstract}
We expand the Chebyshev polynomials and some of its linear combination in
linear combinations of the $q-$Hermite, the Rogers ($q-$utraspherical) and
the Al-Salam--Chihara polynomials and vice versa. We use these expansions to
obtain expansions of some densities, including $q-$Normal and some related
to it, in infinite series constructed of the products of the other density
times polynomials orthogonal to it, allowing deeper analysis and discovering
new properties. On the way we find an easy proof of expansion of the
Poisson--Mehler kernel as well as its reciprocal. We also formulate simple
rule relating one set of orthogonal polynomials to the other given the
properties of the ratio of the respective densities of measures
orthogonalizing these polynomials sets.
\end{abstract}

\thanks{The author would like to thank unknown referee for his tedious work
in pointing out numerous editorial mistakes of the manuscript.}
\maketitle

\section{Introduction}

The aim of this paper is to formulate a simple rule of expanding one density
in terms of products of the other density times the polynomial orthogonal
with respect to this density. Then to present some of its consequences and
applications. The original aim of such expansions was to use them to find
some 'easy to generate', simple densities that bound from above other
densities that were given in the form of infinite products. In other words
the original aim of such expansions was practical and connected with the
idea of generating i.i.d. sequences of observations drawn from distributions
given by the densities that have difficult to analyze form, e.g. are given
in the form of an infinite product. Later however, it turned out that such
expansions are interesting by its own allowing deeper insight into
distributions that are defined by the densities involved. In particular 'two
lines proofs' are possible of the identities that traditionally are proved
on a half or more pages.

A simple reflection leads to the conclusion that we deal with this type of
situation in the case of e.g. the Poisson--Mehler expansion formula or
recently obtained (see \cite{szab2}) expansion of the $q-$Normal density in
terms of products of the Wigner density times appropriately scaled Chebyshev
polynomials. Thus it is the time to generalize it, formulate general rule
and obtain some new expansions. It will turn out that following this general
rule, the difficulty of obtaining expansion of the type discussed in the
paper is shifted to difficulties in getting the so called "connection
coefficients" obtained by expanding one family of orthogonal polynomials
with respect to the other.

In particular we will obtain new expansions of the so called $q-$Conditional
Normal density in the series of Kesten--McKay type density times some
special combinations of Chebyshev polynomials, reciprocals of the
Poisson--Mehler expansion formula and expansions of some other more
specialized densities.

The two mentioned above densities and distributions defined by them appeared
recently in works of Bo\.{z}ejko at al. \cite{Bo} in the noncommutative or
Bryc \cite{bryc1}, \cite{bryc2} in the classical probability context. These
densities were originally defined in terms of infinite products and thus it
was difficult to work with them for someone not familiar with notation,
notions and basic results in the so called $q-$series theory.

In many branches or applications of functional analysis such as theory of
linear operators or quantum groups we deal with Mercer's type kernels i.e.
expressions of the type $\sum_{n\geq 0}r_{n}\phi _{n}\left( x\right) \psi
_{n}\left( y\right) ,$ where $\left\{ r_{n}\right\} _{n\geq 0}$ is a
sequence of reals, $\left\{ \phi _{n}\left( x\right) \right\} _{n\geq 0},$ $%
\left\{ \psi _{n}\left( x\right) \right\} _{n\geq 0}$ are sequences of
square integrable functions from some space $L_{2}\left( \mathbb{R},\mathcal{%
B},\gamma \right) .$ The problem is to provide conditions for non-negativity
of such a kernel when $\left( x,y\right) $ belong to some Cartesian product
of intervals, prove this non-negativity and also express such a kernel in
some compact, easier to analyze, form, i.e. sum it. The point is that many
of the expansions that we have obtained in the paper are in fact certain
kernels. Variable $y$ plays a r\^{o}le of a parameter. By the nature of the
expansion we know its sum and know that it is nonnegative. Hence the paper
can be helpful solving important problems associated with summing and
examining positivity of kernels.

The ideas we are presenting here are universal and can be applied to any
densities and systems of orthogonal polynomials.

The paper is organized as follows. The next Section \ref{idea} presents
general idea of expansion, the main subject of the paper, as well as simple
Proposition presenting relationship between sets of polynomials given the
ratio of the densities of measures orthogonalizing these sets of
polynomials. The task, in a sense, inverse to the idea of expansion. This
section contains also Subsection \ref{przyklady} that presents some,
instructive and believed to be interesting, examples of expansions between
simple measures orthogonalizing well known sets of polynomials such as
Chebyshev, Hermite and some of their combinations. Next we introduce
notation used in the $q-$series theory in Subsection \ref{notation}. Then we
list densities that we will analyze in Subsection \ref{densities}. Finally
we present families of orthogonal polynomials that will be used in the
sequel and associate them with measures that make these sets of polynomials
orthogonal in Subsection \ref{polynomials}. In particular we present here
the $q-$Hermite, the Al-Salam--Chihara and the Rogers ($q-$utraspherical)
polynomials. The next Section \ref{auxi} is devoted to listing known and
finding some new connection coefficients between considered in the paper
families of polynomials. Section \ref{exp} presents main results of the
paper, that is expansions of one density in the series of the other density
times series of polynomials orthogonal with respect to this other measure.
We also give some (by no means all possible) immediate consequences that
lead to interesting identities. Finally Section \ref{dowody} contains some
lengthy proofs of some of the results of Section \ref{auxi}. This section,
following suggestion of the referee, contains also a few sentences
presenting basic properties of orthogonal polynomials as well as reference
to some literature dedicated to the theory of orthogonal polynomials.

\section{Idea of expansion\label{idea}}

The idea of expansion that we are going to pursue is general, simple and is
not new (it can be found in e.g. in \cite{Is2005}, Exercise 2.9). We believe
that it is very fruitful and has not been sufficiently exploited. It is as
follows.

Suppose we have two measures $\alpha $ and $\beta $ defined on $\mathbb{R}.$
Let us define two spaces $L_{2}\left( \mathbb{R},\mathcal{B},\mathcal{\alpha 
}\right) $ and $L_{2}\left( \mathbb{R},\mathcal{B},\mathcal{\beta }\right) ,$
where $\mathcal{B}$ denotes a set of Borel subsets, of real functions
defined on $\mathbb{R}$, square integrable with respect to measures $\alpha $
and $\beta $ respectively. Assume also that $\limfunc{supp}\beta \subseteq 
\limfunc{supp}\alpha $. Further suppose that we know the sets of polynomials 
$\left\{ a_{n}\left( x\right) \right\} _{n\geq 0}$ and $\left\{ b_{n}\left(
x\right) \right\} _{n\geq 0}$ defined on $\mathbb{R}$ that are orthogonal
with respect to the measures $\alpha $ and $\beta $ respectively. That is,
assume that we know that: 
\begin{gather*}
\forall m,n\geq 0:\int_{\mathbb{R}}a_{n}\left( x\right) a_{m}\left( x\right)
d\alpha (x)\allowbreak =\allowbreak \delta _{nm}\hat{a}_{n}, \\
\int_{\mathbb{R}}b_{n}\left( x\right) b_{m}\left( x\right) d\beta \left(
x\right) \allowbreak =\allowbreak \delta _{mn}\hat{b}_{n},
\end{gather*}%
where $\delta _{mn}$ denotes as usually Kronecker's delta.

Suppose also that we know connection coefficients between the sets $\left\{
a_{n}\left( x\right) \right\} _{n\geq 0}$ and $\left\{ b_{n}\left( x\right)
\right\} _{n\geq 0}$ i.e. we know numbers $\gamma _{k,n}$ such that 
\begin{equation*}
\forall n\geq 1:a_{n}\left( x\right) \allowbreak =\allowbreak
\sum_{k=0}^{n}b_{k}\left( x\right) \gamma _{k,n}.
\end{equation*}%
Further suppose that the measures $\alpha $ and $\beta $ have densities $%
A\left( x\right) $ and $B\left( x\right) $ respectively. Then

\begin{equation}
B\left( x\right) =A\left( x\right) \sum_{n=0}^{\infty }c_{n}a_{n}\left(
x\right) ,  \label{rozkl}
\end{equation}%
where $c_{n}\allowbreak =\allowbreak \gamma _{0,n}\hat{b}_{0}/\hat{a}_{n}.$

The sense of (\ref{rozkl}) and the type of its convergence depends on the
properties of the functions $B\left( x\right) $, $A\left( x\right) $ and the
coefficients $\left\{ c_{n}\right\} _{n\geq 1}.$ If 
\begin{equation*}
\int_{\mathbb{R}}\left( B\left( x\right) ^{2}/A^{2}\left( x\right) \right)
d\alpha (x)\allowbreak <\allowbreak \infty
\end{equation*}%
that is if $B\left( x\right) /A\left( x\right) \allowbreak \in \allowbreak
L_{2}\left( \mathbb{R},\mathcal{B},\mathcal{\alpha }\right) $, series $%
\sum_{n=0}^{\infty }c_{n}a_{n}\left( x\right) $ converges in $L_{2}\left( 
\mathbb{R},\mathcal{B},\mathcal{\alpha }\right) $ and depending on the
coefficients $\left\{ c_{n}\right\} _{n\geq 0}$ we can even have almost
(with respect to $\alpha )$ pointwise convergence (more precisely if $%
\sum_{i\geq 1}\left\vert c_{n}\right\vert ^{2}\log ^{2}n<\infty ,$ by the
Rademacher--Menshov Thm.).

However in general $B\left( x\right) /A\left( x\right) $ is only integrable
with respect to measure $\alpha .$ Then one has to refer to the distribution
theory. $\sum_{n=0}^{\infty }c_{n}a_{n}\left( x\right) $ is then in general
a distribution of order $0$.

To see that really 
\begin{equation*}
c_{n}\allowbreak =\allowbreak \gamma _{0,n}\hat{b}_{0}/\hat{a}_{n},
\end{equation*}%
for $n\geq 0$ let us multiply both sides of (\ref{rozkl}) by $\alpha
_{n}\left( x\right) $ and integrate over $\limfunc{supp}\alpha $. On the
left hand side we will get $\gamma _{0,n}\hat{b}_{0}$ since 
\begin{equation*}
\int_{\mathbb{R}}b_{k}\left( x\right) B\left( x\right) dx\allowbreak
=\allowbreak 0
\end{equation*}%
for $k\geq 1.$ On the right hand side we get $c_{n}\hat{a}_{n}.$

\begin{remark}
Of course to get the expansion (\ref{rozkl}) one needs only to calculate 
\begin{equation*}
\int_{\mathbb{R}}\alpha _{m}\left( x\right) d\beta \left( x\right)
\allowbreak =\allowbreak \gamma _{0,m}.
\end{equation*}
On the other hand to get connection coefficients one needs to do some
algebra without integration. This sometimes can be simpler.
\end{remark}

The idea of relating sets of polynomials given the relationship between
measures that make these sets of polynomials orthogonal is not new (see e.g. 
\cite{Is2005} Thm.2.7.1 (by Christoffel)), assertion iii). Christoffel's
relationship between sets of polynomials given the fact that the ratio
between orthogonalizing these polynomials measures is a polynomial is
accurate given the zeros of this polynomial. If the polynomial is of order
more than $2$ it is hard to find these zeros as functions of coefficients.
This is of course limitation of possible applications of Christoffel's
result. The following simple Proposition can be viewed as simplified
modification of Christoffel's Theorem. It contains series of simple remarks
concerning relationships between discussed sets of polynomials. They do not
give precise relationship but in particular situation, confronted together
can give such connection. Besides here the only thing one has to know about
the ratio of the measures is its expansion with respect to one of these sets
of polynomials.

\begin{proposition}
\label{iloraz}Suppose $\alpha $, $\beta ,$ $A\left( x\right) ,$ $B\left(
x\right) ,$ are as described above. Assume also that $\limfunc{supp}\beta
\subseteq \limfunc{supp}\alpha .$ Suppose further that $\left\{
a_{i}\right\} _{i\geq 1}$ and $\left\{ b_{i}\right\} _{i\geq 1}$ polynomials
are monic\footnote{%
Polynomial $p_{n}\left( x\right) $ of order $n$ is called monic if
coefficient at $x^{n}$ is equal to $1.$}. Suppose additionally that we know
that $B\left( x\right) /A\left( x\right) \allowbreak =\allowbreak W\left(
x\right) ,$where $W$ can be expanded in the series of polynomials $%
a_{i}\left( x\right) :$ 
\begin{equation*}
W\left( x\right) \allowbreak =\allowbreak 1+\sum_{i=1}^{N}w_{i}a_{i}\left(
x\right) /\hat{a}_{i}
\end{equation*}%
where $\hat{a}_{i}\allowbreak =\allowbreak \int a_{i}^{2}\left( x\right)
A\left( x\right) dx,$ converging in $L_{2}\left( \mathbb{R}\mathbf{,}%
\mathcal{B},\mathcal{\alpha }\right) $. Put $w_{0}\allowbreak =\allowbreak
1. $ Number $N$ can be finite or infinite. Let us recursively define the
sequence of numbers $\left\{ f_{n}\right\} _{n\geq 0},$ with $%
f_{0\allowbreak }\allowbreak =\allowbreak 1$ by: 
\begin{equation*}
n\geq 1:\sum_{i=0}^{n}f_{n-i}w_{i}\allowbreak =\allowbreak 0,
\end{equation*}%
where we set $w_{i}\allowbreak =\allowbreak 0$ for $i\geq N\allowbreak
+\allowbreak 1$ if $N$ is finite.

i) Then monic polynomials defined by: 
\begin{equation*}
\phi _{n}\left( x\right) \allowbreak =\sum_{i=0}^{n}f_{n-i}a_{i}\left(
x\right)
\end{equation*}%
satisfy $\int_{\mathbb{R}}\phi _{n}\left( x\right) B\left( x\right)
dx\allowbreak =\allowbreak 0,$ $n=1,2,\ldots $ . Besides for $\forall n\geq
1 $: 
\begin{equation*}
a_{n}\left( x\right) \allowbreak =\allowbreak \sum_{i=0}^{n}w_{n-i}\phi
_{i}\left( x\right) .
\end{equation*}

ii) If $N$ is finite, then $\int a_{i}\left( x\right) dB\left( x\right)
\allowbreak =\allowbreak w_{i},$ $i\allowbreak =\allowbreak 1,\ldots ,N,$
and $\int a_{i}\left( x\right) dB\left( x\right) \allowbreak =\allowbreak
0,~\forall i\geq N+1.$ In particular:%
\begin{equation*}
a_{n}\left( x\right) \allowbreak =\allowbreak \phi _{n}\left( x\right)
+\sum_{i=1}^{N}w_{i}\phi _{n-i}\left( x\right) ,
\end{equation*}%
for $n\geq N+1.$

iii) If $N$ is finite then there exist $N$ sequences $\left\{ \gamma
_{n,j}\right\} _{n\geq 1,1\leq j\leq N}$ such that $\forall n\geq 1$%
\begin{equation*}
a_{n}\left( x\right) \allowbreak =\allowbreak b_{n}\left( x\right)
\allowbreak +\allowbreak \sum_{j=1}^{N}\gamma _{n,j}b_{n-j}\left( x\right) .
\end{equation*}
\end{proposition}

\begin{proof}
Is moved to section \ref{dowody}
\end{proof}

\begin{remark}
The most important assertion of the Proposition above is the assertion iii).
It is illustrated by at least two examples presented below: Example 1 where
we analyze the ratio of the two densities with respect to which the
Chebyshev polynomials of the second and the first kind are orthogonal. This
ratio is a polynomial of order $2$ ( $N=2)$ and thus we have formula (\ref%
{tnau}) expressing the Chebyshev polynomials of the first kind as a finite
(involving $3\allowbreak =\allowbreak N+1$ last only) combination of the
Chebyshev polynomials of the second kind. Similar situation is in Example 2,
below.
\end{remark}

\subsection{Examples\label{przyklady}}

Let us denote as usually: $I_{A}\left( x\right) =\left\{ 
\begin{array}{ccc}
1 & if & x\in A \\ 
0 & if & x\notin A%
\end{array}%
\right. $.

\begin{enumerate}
\item Let us take 
\begin{equation*}
A\left( x\right) \allowbreak =\allowbreak \frac{1}{\pi \sqrt{1-x^{2}}}%
I_{(-1,1)}\left( x\right) ,\text{and }a_{n}\left( x\right) \allowbreak
=\allowbreak T_{n}\left( x\right) ,n\geq -1
\end{equation*}%
(Chebyshev Polynomials of the first kind). Further let us take: 
\begin{equation*}
B\left( x\right) \allowbreak =\allowbreak \frac{2}{\pi }\sqrt{1-x^{2}}%
I_{\left( -1,1\right) }\left( x\right) ,\text{ }b_{n}\left( x\right)
\allowbreak =U_{n}\left( x\right) ,n\geq -1
\end{equation*}%
(Chebyshev Polynomials of the second kind). It is (see e.g. \cite{AAR} or 
\cite{Is2005}) known that 
\begin{eqnarray*}
\int_{-1}^{1}a_{n}\left( x\right) a_{m}\left( x\right) A\left( x\right)
dx\allowbreak &=&\left\{ 
\begin{array}{ccc}
1 & for & n=m=0 \\ 
\frac{1}{2}\delta _{nm} & for & n\neq 0\text{ or }m\neq 0%
\end{array}%
\right. , \\
\int_{-1}^{1}b_{n}\left( x\right) b_{m}\left( x\right) B\left( x\right)
dx\allowbreak &=&\allowbreak \delta _{nm}.
\end{eqnarray*}%
Polynomials $\left\{ T_{n}\right\} $ and $\left\{ U_{n}\right\} $ satisfy
the same three term recurrence however with different initial conditions for 
$n\allowbreak =\allowbreak 1.$ Namely $T_{-1}\left( x\right) \allowbreak
\allowbreak =\allowbreak U_{-1}\left( x\right) \allowbreak =0,$ $\allowbreak
T_{0}\left( x\right) \allowbreak =\allowbreak U_{0}\left( x\right)
\allowbreak =\allowbreak 1,$ $T_{1}\left( x\right) \allowbreak =\allowbreak
x,$ $U_{1}\left( x\right) \allowbreak =\allowbreak 2x$ and 
\begin{equation}
2xT_{n}\left( x\right) =T_{n+1}\left( x\right) +T_{n-1}\left( x\right) ,
\label{_0}
\end{equation}
for $n\geq 0.$

Now notice that 
\begin{equation*}
(U_{1}\left( x\right) \allowbreak -\allowbreak U_{-1}\left( x\right)
)/2\allowbreak =\allowbreak x\allowbreak =\allowbreak T_{1}\left( x\right) .
\end{equation*}%
Besides we have 
\begin{gather*}
x(U_{n}\left( x\right) \allowbreak -U_{n-2}\left( x\right) )/2\allowbreak
=\allowbreak (U_{n+1}\left( x\right) \allowbreak \allowbreak +\allowbreak
U_{n-1}\left( x\right) -U_{n-1}\left( x\right) \allowbreak +\allowbreak
U_{n-3}\left( x\right) )/2\allowbreak \\
=\allowbreak (U_{n+1}\left( x\right) \allowbreak -\allowbreak U_{n-1}\left(
x\right) )/2+(U_{n-1}\left( x\right) \allowbreak -\allowbreak U_{n-3}\left(
x\right) )/2,
\end{gather*}%
which is the three term recurrence (\ref{_0}) satisfied by polynomials $%
T_{n}.$ Hence: 
\begin{equation}
\forall n\geq 1:T_{n}\left( x\right) \allowbreak =\allowbreak \left(
U_{n}\left( x\right) -U_{n-2}\left( x\right) \right) /2.  \label{tnau}
\end{equation}%
Thus consequently we have $\gamma _{0,0}\allowbreak =\allowbreak 1,$ $%
\allowbreak \allowbreak $%
\begin{equation*}
\gamma _{k,n}\allowbreak =\allowbreak \left\{ 
\begin{array}{ccc}
1/2 & if & k=n \\ 
-1/2 & if & k=n-2 \\ 
0 & if & otherwise%
\end{array}%
\right. ,
\end{equation*}%
for $n\geq 1$. So $\gamma _{0,0}\allowbreak =\allowbreak 1,$ $\gamma
_{0,1}\allowbreak =\allowbreak 0,\gamma _{0,2}\allowbreak =\allowbreak -1/2,$
$\gamma _{0,n}\allowbreak =\allowbreak 0$ for $n\geq 3.$ Hence we have
elementary relationship 
\begin{equation*}
B\left( x\right) \allowbreak =\allowbreak A\left( x\right) (1-T_{2}\left(
x\right) )\allowbreak =\allowbreak 2A\left( x\right) (1-x^{2}).
\end{equation*}

Similarly one can deduce that : 
\begin{equation*}
\forall n\geq 1:U_{n}\left( x\right) =2\sum_{i=0}^{\left\lfloor
n/2\right\rfloor }T_{n-2i}\left( x\right) -\left( 1+\left( -1\right)
^{n}\right) /2.
\end{equation*}%
Hence $\gamma _{0,2i+1}\allowbreak =\allowbreak 0,$ $\gamma
_{0,2i}\allowbreak =\allowbreak 1,$ $i=0,1,2,\ldots $ $.$ Thus we have 
\begin{equation*}
A\left( x\right) \allowbreak =\allowbreak B\left( x\right)
\sum_{i=0}^{\infty }U_{2i}\left( x\right)
\end{equation*}%
and we do not have neither pointwise nor even $\func{mod}$ $\beta $
convergence\footnote{%
'$\func{mod}$ $\beta $' traditionally in probability means 'in measure $%
\beta $'$.$}. One can deduce, following definition of distributions that the
right hand side of the above equality is a distribution $t_{\alpha }$ for
which $\forall n\geq 1$ $\ t_{\alpha }\left( T_{n}\right) \allowbreak
=\allowbreak 0,$ by (\ref{tnau}) and orthogonality of $\left\{ U_{i}\right\}
_{i\geq 1}$ with respect to $B\left( x\right) .$ However we are not going to
continue this topic since our main concern are regular, convergent cases.
Deeper analysis as well as the generalization of this case can lead to some
interesting theoretical problems. In particular what is the meaning of
similar expansions in the case when the condition $\limfunc{supp}\beta
\subseteq \limfunc{supp}\alpha $ is not satisfied but the connection
coefficients are known?

\item Let%
\begin{equation*}
A\left( x|y,\rho \right) \allowbreak =\allowbreak \frac{\left( 1-\rho
^{2}\right) \sqrt{4-x^{2}}}{2\pi ((1-\rho ^{2})^{2}-\rho xy(1+\rho
^{2})+\rho ^{2}(x^{2}+y^{2}))}
\end{equation*}%
if $x\in (-2,2)$ and $0$ otherwise and $\left\vert y\right\vert \leq 2,$ $%
\left\vert \rho \right\vert <1$ be a particular case of the Kesten--McKay
density considered also in the sequel. It is known (also it follows the fact
that it is a particular case of considered in the sequel distribution $%
f_{CN})$ that the following polynomials 
\begin{equation*}
k_{n}\left( x|y,\rho \right) \allowbreak =\allowbreak U_{n}\left( x/2\right)
\allowbreak -\allowbreak \rho yU_{n-1}\left( x/2\right) \allowbreak
+\allowbreak \rho ^{2}U_{n-2}\left( x/2\right)
\end{equation*}%
when $n\geq 2,$ $k_{1}\left( x|y,\rho \right) \allowbreak =\allowbreak
x-\rho y$ and $k_{0}\left( x|y,\rho \right) \allowbreak =\allowbreak 1$ are
orthogonal with respect to the measure defined by $A.$

As the measure $\beta $ let us take same measure as in the previous example
but re-scaled by $2$. More precisely let $\beta $ have density 
\begin{equation*}
B\left( x\right) \allowbreak =\allowbreak \frac{1}{2\pi }\sqrt{4-x^{2}}.
\end{equation*}%
Hence re-scaled Chebyshev polynomials $U_{n}\left( x/2\right) $ are
orthogonal with respect to $\beta .$ As far as the expansion of $B$ is
concerned we have 
\begin{equation*}
\gamma _{0,n}\allowbreak =\allowbreak \left\{ 
\begin{array}{ccc}
0 & if & n>2 \\ 
\rho ^{2} & if & n=2 \\ 
-\rho y & if & n=1%
\end{array}%
.\right.
\end{equation*}%
Besides it is known (also from (\ref{z p})) that $\int_{-2}^{2}k_{n}^{2}%
\left( x|y,\rho ,0\right) A\left( x|y,\rho \right) \allowbreak =\allowbreak
(1-\rho ^{2}).$ Hence we have:%
\begin{eqnarray*}
B\left( x\right) \allowbreak &=&\allowbreak A\left( x|y,\rho \right) \left(
1-\frac{\rho y}{(1-\rho ^{2})}k_{1}\left( x|y,\rho \right) +\frac{\rho ^{2}}{%
\left( 1-\rho ^{2}\right) }k_{2}\left( x|y,\rho \right) \right) \\
&=&A\left( x|y,\rho \right) \left( (1-\rho ^{2})^{2}-\rho (1-q)xy(1+\rho
^{2})+(1-q)\rho ^{2}(x^{2}+y^{2})\right) /(1-\rho ^{2}).
\end{eqnarray*}%
On the other hand one can easily derive (or it follows from (4.7) in \cite%
{IRS99} considered for $q\allowbreak =\allowbreak 0$ and noting that $%
h_{n}\left( x|0\right) \allowbreak =\allowbreak U_{n}\left( x\right) $) that 
\begin{equation*}
U_{n}(x/2)\allowbreak =\allowbreak \sum_{j=0}^{n}\rho ^{n-j}U_{n-j}\left(
y/2\right) k_{j}\left( x|y,\rho \right) .
\end{equation*}%
Thus we have $\gamma _{0,n}\allowbreak =\allowbreak \rho ^{n}U_{n}\left(
y/2\right) $ and consequently: 
\begin{equation}
A\left( x|y,\rho \right) \allowbreak =\allowbreak B\left( x\right)
\sum_{i=0}^{\infty }\rho ^{i}U_{i}\left( y/2\right) U_{i}\left( x/2\right) ,
\label{PMq=0}
\end{equation}%
which is a particular case of the Poisson --Mehler kernel to be discussed in
the sequel. \newline

\item Following well known (see e.g. \cite{AAR} Ex. 5, p. 339) formula
concerning Hermite polynomials $H_{n}$ orthogonal with respect to the
measure 
\begin{equation*}
d\alpha \left( x\right) \allowbreak =\allowbreak \frac{1}{\sqrt{2\pi }}\exp
\left( -x^{2}/2\right) dx\allowbreak \overset{df}{=}\allowbreak A\left(
x\right) dx,
\end{equation*}%
\begin{equation*}
\forall \rho \in \left( -1,1\right) ,\forall n\geq 1:H_{n}\left( \rho x+y%
\sqrt{1-\rho ^{2}}\right) =\sum_{i=0}^{n}\binom{n}{i}\rho ^{i}\left( \sqrt{%
1-\rho ^{2}}\right) ^{n-i}H_{i}\left( x\right) H_{n-i}\left( y\right) ,
\end{equation*}%
we can rewrite it in the following form :%
\begin{equation*}
\forall \rho \in \left( -1,1\right) ,\forall n\geq 1:H_{n}\left( x\right)
=\sum_{i=0}^{n}\binom{n}{i}\rho ^{i}H_{i}\left( y\right) \left( \sqrt{1-\rho
^{2}}\right) ^{n-i}H_{n-i}\left( \frac{(x-\rho y)}{\sqrt{1-\rho ^{2}}}%
\right) ,
\end{equation*}%
since we have trivially 
\begin{equation*}
x\allowbreak =\allowbreak \rho y+\sqrt{1-\rho ^{2}}\frac{(x-\rho y)}{\sqrt{%
1-\rho ^{2}}},
\end{equation*}%
and view it as a 'connection coefficient formula' between sets of
polynomials $\left\{ H_{n}\left( x\right) \right\} _{n\geq 0}$ that are
orthogonal with respect the measure $d\alpha $ and\newline
$\left\{ \left( \sqrt{1-\rho ^{2}}\right) ^{n}H_{n}\left( \frac{(x-\rho y)}{%
\sqrt{1-\rho ^{2}}}\right) \right\} _{n\geq 0}$that are orthogonal with
respect to the measure 
\begin{equation*}
d\beta \left( x\right) \allowbreak =\allowbreak \frac{1}{\sqrt{2\pi (1-\rho
^{2})}}\exp \left( -\frac{\left( x-\rho y\right) ^{2}}{2\left( 1-\rho
^{2}\right) }\right) dx\allowbreak \overset{df}{=}\allowbreak B\left(
x\right) dx.
\end{equation*}%
An easy calculation gives $\gamma _{0,n}\allowbreak =\allowbreak \rho
^{n}H_{n}\left( y\right) $ and $\hat{a}\allowbreak =\allowbreak n!$ and we
end up with famous Mehler Hermite Polynomial Formula%
\begin{equation}
\frac{1}{\sqrt{2\pi (1-\rho ^{2})}}\exp \left( -\frac{\left( x-\rho y\right)
^{2}}{2\left( 1-\rho ^{2}\right) }\right) \allowbreak =\allowbreak \frac{1}{%
\sqrt{2\pi }}\exp \left( -\frac{x^{2}}{2}\right) \sum_{i=0}^{\infty }\frac{%
\rho ^{i}}{i!}H_{i}\left( x\right) H_{i}\left( y\right) ,  \label{MHPF}
\end{equation}%
which is better known in a form obtained from the above by dividing both
sides by $\allowbreak \frac{1}{\sqrt{2\pi }}\exp \left( -\frac{x^{2}}{2}%
\right) $ and whose proof takes about a page in popular handbooks of special
functions like e.g. \cite{AAR}.
\end{enumerate}

\section{Densities and families of orthogonal polynomials. Their properties
and relationships\label{properties}}

\subsection{Notation\label{notation}}

We will use traditional notation of the $q-$series theory i.e. 
\begin{equation*}
\left[ 0\right] _{q}\allowbreak =\allowbreak 0;\left[ n\right]
_{q}\allowbreak =\allowbreak 1+q+\ldots +q^{n-1}\allowbreak ,\left[ n\right]
_{q}!\allowbreak =\allowbreak \prod_{i=1}^{n}\left[ i\right] _{q},
\end{equation*}%
with $\left[ 0\right] _{q}!\allowbreak =\allowbreak 1,$%
\begin{equation*}
\QATOPD[ ] {n}{k}_{q}\allowbreak =\allowbreak \left\{ 
\begin{array}{ccc}
\frac{\left[ n\right] _{q}!}{\left[ n-k\right] _{q}!\left[ k\right] _{q}!} & 
when & n\geq k\geq 0 \\ 
0 & when & otherwise%
\end{array}%
\right. .
\end{equation*}%
Sometimes it will be useful to use the so called $q-$Pochhammer symbol : 
\begin{equation*}
\forall n\geq 1:\left( a;q\right) _{n}=\prod_{i=0}^{n-1}\left(
1-aq^{i}\right) ,
\end{equation*}%
with $\left( a;q\right) _{0}=1$ , 
\begin{equation*}
\left( a_{1},a_{2},\ldots ,a_{k};q\right) _{n}\allowbreak =\allowbreak
\prod_{i=1}^{k}\left( a_{i};q\right) .
\end{equation*}%
It is easy to notice that $\left( q;q\right) _{n}=\left( 1-q\right) ^{n}%
\left[ n\right] _{q}!$ and that 
\begin{equation*}
\QATOPD[ ] {n}{k}_{q}\allowbreak =\allowbreak \left\{ 
\begin{array}{ccc}
\frac{\left( q;q\right) _{n}}{\left( q;q\right) _{n-k}\left( q;q\right) _{k}}
& when & n\geq k\geq 0 \\ 
0 & when & otherwise%
\end{array}%
\right. .
\end{equation*}%
Notice also that $\left( a;0\right) _{n}\allowbreak =\allowbreak 1-a$ for $%
n\geq 1$ and $\left( a;1\right) _{n}\allowbreak =\allowbreak \left(
1-a\right) ^{n}.$

If it will not cause misunderstanding Pochhammer symbol $(a;q)_{n}$ or $%
\left( a_{1},a_{2},\ldots ,a_{k};q\right) _{n}$ will often be abbreviated to 
$\left( a\right) _{n}$ and $\left( a_{1},a_{2},\ldots ,a_{k}\right) _{n}$ if
the choice of $q$ is obvious.

Let us also denote $:$ 
\begin{equation*}
S\left( q\right) \allowbreak =\allowbreak \left[ -2/\sqrt{1-q},2/\sqrt{1-q}%
\right] .
\end{equation*}

\subsection{Densities defined by infinite products\label{densities}}

As it follows from the three examples discussed by the end of previous
section, the idea of expanding one density with a help of another, can be
fruitful and lead to interesting formulae and consequently to deeper
understanding of the expanded distribution. Besides some recently used
distributions have densities that are defined with a help of infinite
products. Infinite products are in many ways difficult to deal with. In
particular they are more difficult to calculate many quantities that are
interesting for probabilists like moments for example. That is why we will
use this technique of expansion to accustom, that is to obtain another, more
suitable for further analysis and research, form of the three densities that
appeared recently and that are defined by an infinite series.

Two of these three distributions appeared in the context of one dimensional
random fields (see details \cite{bryc1} and \cite{bms}), $q-$Gaussian
processes (for details see \cite{Bo}) or quadratic harnesses considered by
Bryc at al. \cite{BryWes10}.

All three distributions appeared in the context of special functions in
particular in the context of the Rogers polynomials. However only recently
their importance to both commutative and noncommutative probability became
apparent. As mentioned before distributions $f_{N}\left( x|q\right) $ and $%
f_{CN}\left( x|y,\rho ,q\right) $ that are defined below reappeared in 1997
in the paper \cite{Bo} of Bo\.{z}ejko and Speicher in a purely
noncommutative probability context.

The densities that we are primarily going to analyze are as follows:%
\begin{equation}
f_{N}\left( x|q\right) =\frac{\sqrt{1-q}\left( q\right) _{\infty }}{2\pi 
\sqrt{4-(1-q)x^{2}}}\prod_{k=0}^{\infty }\left(
(1+q^{k})^{2}-(1-q)x^{2}q^{k}\right)  \label{qN}
\end{equation}%
defined for $\left\vert q\right\vert <1$ and $\left\vert x\right\vert
\allowbreak <\allowbreak \frac{2}{\sqrt{1-q}}$ that will be sometimes
referred to as $q-$Normal (briefly $q-$N) distribution and 
\begin{subequations}
\label{fCN}
\begin{gather}
f_{CN}\left( x|y,\rho ,q\right) =\frac{\sqrt{1-q}\left( \rho ^{2},q\right)
_{\infty }}{2\pi \sqrt{4-(1-q)x^{2}}}\times  \label{1} \\
\prod_{k=0}^{\infty }\frac{\left( (1+q^{k})^{2}-(1-q)x^{2}q^{k}\right) }{%
(1-\rho ^{2}q^{2k})^{2}-(1-q)\rho q^{k}(1+\rho ^{2}q^{2k})xy+(1-q)\rho
^{2}(x^{2}+y^{2})q^{2k}},  \label{2}
\end{gather}

defined for $\left\vert q\right\vert <1,$ $\left\vert \rho \right\vert <1$ , 
$\left\vert x\right\vert ,\left\vert y\right\vert \allowbreak <\allowbreak 
\frac{2}{\sqrt{1-q}}$ that will be referred to as $(y,\rho ,q)-$Conditional
Normal, (briefly $(y,\rho ,q)-CN$) distribution.

The third one it is the so called $q-$utraspherical density (density with
respect to which Rogers (called also $q-$utraspherical ) polynomials are
orthogonal). It is given by 
\end{subequations}
\begin{equation}
f_{R}\left( x|\beta ,q\right) \allowbreak =\allowbreak \frac{\sqrt{1-q}%
\left( \beta ^{2},q\right) _{\infty }}{2\pi \sqrt{4-(1-q)x^{2}}\left( \beta
,\beta q\right) _{\infty }}\prod_{k=0}^{\infty }\frac{\left(
(1+q^{k})^{2}-(1-q)x^{2}q^{k}\right) }{\left( (1+\beta q^{k})^{2}-(1-q)\beta
x^{2}q^{k}\right) }.  \label{fR}
\end{equation}%
defined also for $\left\vert q\right\vert <1$ and $\left\vert x\right\vert <%
\frac{2}{\sqrt{1-q}}$ and $\left\vert \beta \right\vert <1$. This
distributions is closely related to distributions \ $q-$N and \ $\left(
y,\rho ,q\right) -$CN. Namely we have the following.

\begin{remark}
\label{fRifN}Note that we have 
\begin{equation}
f_{R}\left( x|\beta ,q\right) \allowbreak =\allowbreak f_{N}\left(
x|q\right) \allowbreak \times \allowbreak \frac{\left( \beta ^{2}\right)
_{\infty }}{\left( \beta ,\beta q\right) _{\infty }\prod_{k=0}^{\infty
}\left( (1+\beta q^{k})^{2}-(1-q)\beta x^{2}q^{k}\right) }.  \label{fRnafN}
\end{equation}%
We also have $f_{CN}\left( x|x,\rho ,q\right) \allowbreak =\allowbreak
f_{R}\left( x|\rho ,q\right) /(1-\rho ),$ since 
\begin{equation*}
(1-\rho ^{2}q^{2k})^{2}\allowbreak -\allowbreak (1-q)\rho q^{k}(1+\rho
^{2}q^{2k})x^{2}\allowbreak +\allowbreak 2(1-q)\rho
^{2}x^{2}q^{2k}\allowbreak =\allowbreak \left( 1-\rho q^{k}\right)
^{2}(\left( 1+\rho q^{k}\right) ^{2}\allowbreak -\allowbreak (1-q)\rho
x^{2}q^{k})
\end{equation*}%
and 
\begin{equation*}
\lim_{\beta \rightarrow 1^{-}}f_{R}\left( x|\beta ,q\right) \allowbreak
=\allowbreak \frac{\sqrt{1-q}}{2\pi \sqrt{4-(1-q)x^{2}}}.
\end{equation*}
\end{remark}

\subsection{Polynomials\label{polynomials}}

Recall that every family of orthogonal polynomials is defined by $3-$ term
recursive relationship. The three families of orthogonal polynomials that
appear in connection with densities $f_{N},$ $f_{CN},$ $f_{R}$ are defined
by the following recursive relationships: 
\begin{eqnarray}
H_{n+1}\left( x|q\right) &=&xH_{n}\left( x|q\right) -\left[ n\right]
_{q}H_{n-1}\left( x|q\right) ,  \label{H} \\
R_{n+1}\left( x|\beta ,q\right) &=&\left( 1-\beta q^{n}\right) xR_{n}\left(
x|\beta ,q\right)  \label{R} \\
&&-\left( 1-\beta ^{2}q^{n-1}\right) \left[ n\right] _{q}R_{n-1}\left(
x|\beta ,q\right) ,  \notag \\
P_{n+1}\left( x|y,\rho ,q\right) &=&(x-\rho yq^{n})P_{n}(x|y,\rho ,q)
\label{Al} \\
&&-(1-\rho ^{2}q^{n-1})[n]_{q}P_{n-1}(x|y,\rho ,q),  \notag
\end{eqnarray}%
with $H_{-1}\left( x|q\right) \allowbreak =\allowbreak R_{-1}\left( x|\beta
,q\right) \allowbreak =\allowbreak P_{-1}\left( x|y,\rho ,q\right)
\allowbreak =\allowbreak 0,$ $H_{0}\left( x|q\right) \allowbreak
=\allowbreak R_{0}\left( x|\beta ,q\right) \allowbreak =\allowbreak
P_{0}\left( x|y,\rho ,q\right) \allowbreak =\allowbreak 1$.

Here parameters $\beta ,\rho ,y,q$ have the following bounds: $\left\vert
q\right\vert \leq 1,$ $\left\vert \beta \right\vert ,\left\vert \rho
\right\vert <1,$ $y\in \mathbb{R}.$ The family (\ref{H}) will be referred to
as the family of $q-$Hermite polynomials, family (\ref{R}) will be referred
to as the family of Rogers polynomials. Finally family (\ref{Al}) will be
referred to as the family of Al-Salam--Chihara polynomials.

In fact in the literature (see e.g. \cite{AAR}) more popular are these
families transformed. Namely as the $q-$Hermite polynomials often function
polynomials 
\begin{equation}
h_{n}\left( x|q\right) \allowbreak =\allowbreak \left( 1-q\right)
^{n/2}H_{n}\left( \frac{2x}{\sqrt{1-q}}|q\right) ,n\geq 1  \label{contH}
\end{equation}%
often called also continuous $q-$Hermite polynomials. As Rogers polynomials
function polynomials: 
\begin{equation}
C_{n}\left( x|\beta ,q\right) \allowbreak =\allowbreak \left( q\right)
_{n}\left( 1-q\right) ^{n/2}R_{n}\left( \frac{2x}{\sqrt{1-q}}|\beta
,q\right) ,n\geq 1.  \label{q-US}
\end{equation}%
Finally as Al-Salam--Chihara polynomials function polynomials: 
\begin{equation}
p_{n}(x|a,b,q)\allowbreak =\allowbreak \left( 1-q\right) ^{n/2}P_{n}\left( 
\frac{2x}{\sqrt{1-q}}|\frac{2a}{\sqrt{\left( 1-q\right) b}},\sqrt{b}%
,q\right) ,  \label{contASC}
\end{equation}%
for $\left\vert \beta \right\vert <1,a^{2}>b\geq 0$ or even (see e.g. \cite%
{Is2005}): 
\begin{equation}
Q_{n+1}\left( x|a,b,q\right) \allowbreak =\allowbreak
(2x-(a+b)q^{n})Q_{n}\left( x|a,b,q\right)
-(1-abq^{n-1})(1-q^{n})Q_{n-1}(x|a,b,q),  \label{AlSC1}
\end{equation}%
with $Q_{-1}\left( x|a,b,q\right) \allowbreak =\allowbreak 0,$ $Q_{0}\left(
x|a,b,q\right) \allowbreak =\allowbreak 1$ related to polynomials $P_{n}$ by
:%
\begin{equation*}
P_{n}\left( x|y,\rho ,q\right) \allowbreak =\allowbreak Q_{n}\left( x\sqrt{%
1-q}/2|\frac{\sqrt{1-q}}{2}\rho (y\allowbreak -\allowbreak i\sqrt{\frac{4}{%
1-q}-y^{2}}),\frac{\sqrt{1-q}}{2}\rho (y\allowbreak +\allowbreak i\sqrt{%
\frac{4}{1-q}-y^{2}}),q\right) /\left( 1-q\right) ^{n/2}.
\end{equation*}%
For our purposes, closely connected with probability, the families defined
by (\ref{H}), (\ref{R}) and (\ref{Al}) are more suitable.

The families of polynomials $\left\{ H_{n}\right\} _{\geq -1},\left\{
P_{n}\right\} _{n\geq -1}$ and $\left\{ R_{n}\right\} _{\geq -1}$ have the
following basic properties:

\begin{lemma}
$\forall -1<q\leq 1,\left\vert \rho \right\vert ,\left\vert \beta
\right\vert <1,y\in S\left( q\right) $ we have $\forall n\geq 0$:%
\begin{equation}
\int_{S\left( q\right) }H_{n}\left( x|q\right) H_{m}\left( x|q\right)
f_{N}\left( x|q\right) dx=\left\{ 
\begin{array}{ccc}
0 & when & n\neq m \\ 
\left[ n\right] _{q}! & when & n=m%
\end{array}%
\right. ,  \label{z H}
\end{equation}

\begin{equation}
\int_{S\left( q\right) }H_{n}\left( x|q\right) f_{CN}\left( x|y,\rho
,q\right) dx=\rho ^{n}H_{n}\left( y|q\right) ,  \label{z H2}
\end{equation}

\begin{equation}
\int_{S\left( q\right) }P_{n}\left( x|y,\rho ,q\right) P_{m}\left( x|y,\rho
,q\right) f_{CN}\left( x|y,\rho ,q\right) dx=\left\{ 
\begin{array}{ccc}
0 & when & n\neq m \\ 
\left( \rho ^{2}\right) _{n}\left[ n\right] _{q}! & when & n=m%
\end{array}%
\right. ,  \label{z p}
\end{equation}

\begin{equation}
\int_{S\left( q\right) }R_{n}\left( x|\beta ,q\right) R_{n}\left( x|\beta
,q\right) f_{R}\left( x|\beta ,q\right) dx\allowbreak =\allowbreak \left\{ 
\begin{array}{ccc}
0 & when & n\neq m \\ 
\frac{\left( 1-\beta \right) \left( \beta ^{2}\right) _{n}[n]_{q}!}{\left(
1-\beta q^{n}\right) } & when & n=m%
\end{array}%
\right. ,  \label{z R}
\end{equation}

\begin{equation}
\forall \left\vert \rho _{1}\right\vert ,\left\vert \rho _{2}\right\vert
<1:\int_{S\left( q\right) }f_{CN}\left( x|y,\rho _{1},q\right) f_{CN}\left(
y|z,\rho _{2},q\right) dy=f_{CN}\left( x|z,\rho _{1}\rho _{2},q\right) .
\label{chapman}
\end{equation}

\begin{equation}
\max_{x\in S\left( q\right) }\left\vert H_{n}\left( x|q\right) \right\vert
\leq \frac{W_{n}\left( q\right) }{(1-q)^{n/2}},\max_{x\in S\left( q\right)
}\left\vert R_{n}\left( x|\beta ,q\right) \right\vert \leq \frac{V_{n}\left(
q,\beta \right) }{\left( q\right) _{n}(1-q)^{n/2}},  \label{maksima}
\end{equation}%
where 
\begin{equation}
W_{n}\left( q\right) \allowbreak =\allowbreak \sum_{i=0}^{n}\QATOPD[ ] {n}{i}%
_{q},V_{n}\left( q,\beta \right) \allowbreak =\allowbreak \sum_{i=0}^{n}%
\frac{\left( \beta |q\right) _{i}\left( \beta |q\right) _{n-i}}{\left(
q|q\right) _{i}\left( q|q\right) _{n-i}}.  \label{WandV}
\end{equation}
\end{lemma}

\begin{proof}
(\ref{z H}), follows from (13.1.11) of \cite{Is2005} after necessary
normalization, \ref{z H2} is given in \cite{bryc1}, or \cite{bms}, however
can be also deduced from (\ref{PM}) below. To prove (\ref{z p}) and (\ref{z
R}) one can use \cite{AAR} or \cite{Is2005} where these formulae are proved
with different normalization (in fact for polynomial $Q_{n}$ and $C_{n}$
defined by (\ref{AlSC1} and (\ref{q-US} respectively)). Below we show it in
an elementary way using standard knowledge on orthogonal polynomials and
formulae (\ref{Al}) and (\ref{R}). Firstly let us denote 
\begin{eqnarray*}
A_{n}\allowbreak &=&\allowbreak \int_{S\left( q\right) }R_{n}^{2}(x|\beta
,q)f_{R}\left( x|\beta ,q\right) dx, \\
B_{n}\allowbreak &=&\allowbreak \int_{S\left( q\right) }P_{n}^{2}\left(
x|y,\rho ,q\right) f_{CN}\left( x|y,\rho ,q\right) dx.
\end{eqnarray*}%
Further we multiply both sides (\ref{Al}) and (\ref{R}) once respectively by 
$P_{n-1}\left( x|y,\rho ,q\right) $ and $R_{n-1}(x|\beta ,q)$ and then by $%
P_{n+1}\left( x|y,\rho ,q\right) $ and $R_{n+1}\left( x|\beta ,q\right) $
and integrate respectively with respect. $f_{CN}$ and $f_{R}$ over $S\left(
q\right) $ obtaining respectively 
\begin{equation*}
\int_{S\left( q\right) }x\allowbreak \allowbreak P_{n}\left( x|y,\rho
,q\right) \allowbreak \allowbreak P_{n-1}\left( x|y,\rho ,q\right)
\allowbreak \allowbreak f_{CN}\left( x|y,\rho ,q\right) dx\allowbreak
=\allowbreak (1-\rho ^{2}q^{n-1})[n]_{q}B_{n-1}
\end{equation*}%
and 
\begin{equation*}
\left( 1-\beta q^{n}\right) \int_{S\left( q\right) }x\allowbreak \allowbreak
R_{n}\left( x|\beta ,q\right) \allowbreak \allowbreak R_{n-1}\left( x|\beta
,q\right) \allowbreak \allowbreak f_{R}\left( x|\beta ,q\right)
dx\allowbreak =\allowbreak (1-\beta ^{2}q^{n-1})\left[ n\right] _{q}A_{n-1}
\end{equation*}%
and then 
\begin{equation*}
B_{n+1}\allowbreak =\allowbreak \int_{S\left( q\right) }xP_{n+1}\left(
x|y,\rho ,q\right) P_{n}\left( x|y,\rho ,q\right) f_{CN}\left( x|y,\rho
,q\right) dx
\end{equation*}%
and 
\begin{equation*}
A_{n+1}\allowbreak =\allowbreak \left( 1-\beta q^{n}\right) \int_{S\left(
q\right) }x\allowbreak \allowbreak R_{n}\left( x|\beta ,q\right) \allowbreak
\allowbreak R_{n+1}\left( x|\beta ,q\right) \allowbreak \allowbreak
f_{R}\left( x|\beta ,q\right) dx.
\end{equation*}%
From these equations we deduce that respectively 
\begin{equation*}
B_{n}\allowbreak \allowbreak =(1-\rho ^{2}q^{n-1})[n]_{q}\allowbreak B_{n-1}
\end{equation*}%
and $\allowbreak $%
\begin{equation*}
A_{n}\allowbreak =\allowbreak \frac{1-\beta q^{n-1}}{1-\beta q^{n}}\left(
1-\beta ^{2}q^{n-1}\right) \left[ n\right] _{q}A_{n-1}
\end{equation*}%
from which follows (\ref{z R}) and (\ref{z p}).

Formula (\ref{chapman}) is taken from \cite{bryc1} and \cite{bryc2}. It can
be also found in \cite{Bo}.\newline
Formula (\ref{maksima}) follows formulae 13.1.10 and 13.2.16 of \cite{Is2005}
and (\ref{contH}) and (\ref{q-US}).
\end{proof}

We will also use the already mentioned Chebyshev polynomials of the first $%
T_{n}\left( x\right) $ defined by $T_{n}\left( \cos \theta \right)
\allowbreak =\allowbreak \cos n\theta $ and second kind $U_{n}\left(
x\right) $ defined by $U_{n}\left( \cos \theta \right) \allowbreak
=\allowbreak \frac{\sin \left( n+1\right) \theta }{\sin \theta }$ and
ordinary (probabilistic) Hermite polynomials $H_{n}\left( x\right) $ i.e.
polynomials orthogonal with respect to $\frac{1}{\sqrt{2\pi }}\exp
(-x^{2}/2).$ Recall that Chebyshev polynomials were defined in Subsection %
\ref{przyklady} Example 1 and satisfy $3-$term recurrence (\ref{_0}), while
polynomials $H_{n}$ satisfy $3-$term recurrence (\ref{_1}) below. 
\begin{equation}
xH_{n}\left( x\right) =H_{n+1}\left( x\right) +nH_{n-1}  \label{_1}
\end{equation}
$H_{0}\left( x\right) \allowbreak =\allowbreak 1,$ $H_{1}\left( x\right)
\allowbreak =\allowbreak x$. Moreover we will be using re-scaled versions of
polynomials $T_{n}$ and $U_{n}$ that is 
\begin{eqnarray*}
\hat{T}_{n}\left( x|q\right) \allowbreak &=&\allowbreak T_{n}\left( x\sqrt{%
1-q}/2\right) /\left( 1-q\right) ^{n/2}, \\
\hat{U}_{n}\left( x|q\right) \allowbreak &=&\allowbreak U_{n}\left( x\sqrt{%
1-q}/2\right) /\left( 1-q\right) ^{n/2}.
\end{eqnarray*}%
These modified polynomials are orthogonal with respect to modified densities
that appear in the context of Chebyshev polynomials. That is we have 
\begin{eqnarray*}
\int_{S\left( q\right) }\hat{U}_{n}\left( x|q\right) \hat{U}_{m}\left(
x|q\right) f_{U}\left( x|q\right) dx\allowbreak &=&\allowbreak \delta _{mn},
\\
\int_{S\left( q\right) }\hat{T}_{n}\left( x|q\right) \hat{T}_{m}\left(
x|q\right) f_{T}\left( x|q\right) dx\allowbreak &=&\allowbreak \delta
_{mn}/2,
\end{eqnarray*}%
if $n\vee m\geq 1$ and $1$ if $n\allowbreak =\allowbreak m\allowbreak
=\allowbreak 0$, where we denoted 
\begin{eqnarray}
f_{U}\left( x|q\right) \allowbreak &=&\allowbreak I_{S\left( q\right)
}\left( x\right) \sqrt{\left( 1-q\right) \left( 4-\left( 1-q\right)
x^{2}\right) }/2\pi ,  \label{fU} \\
f_{T}\left( x|q\right) \allowbreak &=&\allowbreak I_{S\left( q\right)
}\left( x\right) /(\sqrt{\left( 1-q\right) /\left( 4-\left( 1-q\right)
x^{2}\right) }\pi ).  \label{fT}
\end{eqnarray}%
The density $f_{U}$ functions sometimes in the literature as the density of
Wigner distribution with radius $2/\sqrt{1-q}$ or the density of the
semicircle distribution. The density $f_{T}$ is often called the density of
the arcsine distribution.

In the sequel there will also appear distribution $f_{CN}\left( x|y,\rho
,0\right) $ re-scaled in the following way%
\begin{equation}
f_{K}\left( x|y,\rho ,q\right) \allowbreak =\allowbreak \frac{\left( 1-\rho
^{2}\right) \sqrt{1-q}\sqrt{4-(1-q)x^{2}}}{2\pi \left( \left( 1-\rho
^{2}\right) ^{2}-\rho (1-q)\left( 1+\rho ^{2}\right) xy+(1-q)\rho ^{2}\left(
x^{2}+y^{2}\right) \right) }I_{S\left( q\right) }\left( x\right) ,
\label{fK}
\end{equation}%
for $-1<q\leq 1,$ $\left\vert \rho \right\vert <1,$ $y\in S\left( q\right) $%
, that is a particular case of so called Kesten--McKay distribution and
which is nothing else but re-scaled density $A\left( x\right) $ considered
above in Example 2.

We have Proposition that relates cases defined by special values of
parameters to known families of polynomials or distributions:

\begin{proposition}
\label{uwaga}

$1.$ 
\begin{gather*}
f_{CN}\left( x|y,0,q\right) \allowbreak =\allowbreak f_{R}\left(
x|0,q\right) \allowbreak =\allowbreak f_{N}(x|q)\allowbreak =\allowbreak \\
f_{U}\left( x|q\right) \left( q\right) _{\infty }\allowbreak \times
\allowbreak \prod_{k=1}^{\infty }\left( (1+q^{k})^{2}-(1-q)x^{2}q^{k}\right)
,
\end{gather*}

$2.$ $\forall n\geq 0:$%
\begin{gather*}
R_{n}\left( x|0,q\right) \allowbreak =\allowbreak H_{n}\left( x|q\right)
,~~H_{n}\left( x|0\right) =U_{n}\left( x/2\right) , \\
H_{n}\left( x|1\right) \allowbreak =\allowbreak H_{n}\left( x\right)
,~~\lim_{\beta \longrightarrow 1^{-}}\frac{R_{n}\left( x|\beta ,q\right) }{%
\left( \beta \right) _{n}}\allowbreak =\allowbreak 2\frac{T_{n}\left( x\sqrt{%
1-q}/2\right) }{(1-q)^{n/2}},
\end{gather*}

$3.$ $\forall n\geq 0:$%
\begin{eqnarray*}
P_{n}\left( x|x,\rho ,q\right) \allowbreak &=&\allowbreak R_{n}\left( x|\rho
,q\right) \allowbreak ,~~P_{n}\left( x|y,0,q\right) =H_{n}(x|q), \\
P_{n}(x|y,\rho ,1)\allowbreak &=&\allowbreak (1-\rho ^{2})^{n/2}H_{n}\left( 
\frac{x-\rho y}{\sqrt{1-\rho ^{2}}}\right) , \\
P_{n}\left( x|y,\rho ,0\right) \allowbreak &=&\allowbreak U_{n}\left(
x/2\right) \allowbreak -\allowbreak \rho yU_{n-1}\left( x/2\right)
\allowbreak +\allowbreak \rho ^{2}U_{n-2}\left( x/2\right) \allowbreak 
\overset{df}{=}\allowbreak k_{n}\left( x|y,\rho \right) .
\end{eqnarray*}

$4.$ relationship (\ref{z H2}) reduces for $\rho =0$ to relationship (\ref{z
H}) with $m=0,$

$5.$ 
\begin{equation*}
f_{N}\left( x|0\right) =\frac{1}{2\pi }\sqrt{4-x^{2}}I_{<-2,2>}\left(
x\right) ,~~f_{N}\left( x|1\right) \allowbreak =\allowbreak \frac{1}{\sqrt{%
2\pi }}\exp \left( -x^{2}/2\right) ,~~f_{R}\left( x|1,q\right) \allowbreak
=\allowbreak f_{T}\left( x|q\right) ,
\end{equation*}

$6.$%
\begin{equation*}
f_{CN}\left( x|y,\rho ,0\right) \allowbreak =\allowbreak f_{K}\left(
x|y,\rho \right) ,~~f_{CN}\left( x|y,\rho ,1\right) \allowbreak =\allowbreak 
\frac{1}{\sqrt{2\pi (1-\rho ^{2})}}\exp \left( -\frac{\left( x-\rho y\right)
^{2}}{2\left( 1-\rho ^{2}\right) }\right) .
\end{equation*}
\end{proposition}

\begin{proof}
$1.$ is obvious. $2.$ follows observation that (\ref{H}) simplifies to (\ref%
{_0}) and (\ref{_1}) for $q\allowbreak =\allowbreak 0$ and $q\allowbreak
=\allowbreak 1$ respectively while (\ref{R}) simplifies to (\ref{H}). Value $%
\lim_{\beta \longrightarrow 1^{-}}\frac{R_{n}\left( x|\beta ,q\right) }{%
\left( \beta \right) _{n}}$ can be found in \cite{Is2005} , formula 13.2.15. 
\newline
$3.$ First three assertions follow either direct observation in the case of $%
P_{n}\left( x|y,\rho ,0\right) $ or comparison of (\ref{Al}) and (\ref{_1})
considered for substitution $x\longrightarrow (x-\rho y)/\sqrt{1-\rho ^{2}}$
and then multiplication of both sides by $\left( 1-\rho ^{2}\right)
^{(n+1)/2}\ $third assertion follows following observations: $P_{-1}\left(
x|y,\rho ,0\right) \allowbreak =\allowbreak 0,$ $P_{0}\left( x|y,\rho
,0\right) =1,$ $P_{1}\left( x|y,\rho ,0\right) \allowbreak =\allowbreak
x-\rho y$ , $P_{2}\left( x|y,\rho ,0\right) \allowbreak =\allowbreak x\left(
x-\rho y\right) \allowbreak -\allowbreak \left( 1-\rho ^{2}\right) ,$ $%
P_{n+1}\left( x|y,\rho ,0\right) \allowbreak =\allowbreak xP_{n}\left(
x|y,\rho ,0\right) \allowbreak -\allowbreak P_{n-1}\left( x|y,\rho ,0\right) 
$ for $n\geq 2$ which is equation (\ref{_0}). $5.$ and $6.$ Their first
assertions are obvious. Secondly we notice that passing to the limit $%
q\longrightarrow 1^{-}$ and applying $2.$ and $3.$ we obtain well known
relationships defining Hermite polynomials. Hence Hermite polynomials are
orthogonal with respect to the measure defined by $f_{N}\left( x|1\right) $.
Thus distributions defined by $f_{N}$ and $f_{CN}$ tend to normal $N\left(
0,1\right) $ and $N\left( \rho y,\left( 1-\rho ^{2}\right) \right) $
distributions weakly as $q\longrightarrow 1^{-}.$ So it is natural to define 
$f_{N}\left( x|1\right) $ and $f_{CN}\left( x|y,\rho ,q\right) $ as they are
in $5.$ and $6.$
\end{proof}

As suggested in Proposition \ref{uwaga} we will be using notation $%
k_{n}\left( x|y,\rho \right) \allowbreak $ instead $P_{n}\left( x|y,\rho
,0\right) $ which is simpler. Besides we have $k_{0}\left( x|y,\rho \right)
\allowbreak =\allowbreak 1,$ $k_{1}\left( x|y,\rho \right) \allowbreak
=\allowbreak x-\rho y,$ $k_{2}\left( x|y,\rho \right) \allowbreak
=\allowbreak x(x-\rho y\allowbreak )-\allowbreak (1-\rho ^{2})$ and $%
k_{n+1}\left( x|y,\rho \right) \allowbreak =\allowbreak xk_{n}\left(
x|y,\rho \right) \allowbreak -\allowbreak k_{n-1}\left( x|y,\rho \right) .$

\begin{remark}
\label{kesten}Since polynomials $\left\{ k_{n}\left( x|y,\rho \right)
\right\} _{n\geq 0}$ are orthogonal with respect to the measure with density 
$A\left( x\right) $ of Example $2,$ or more precisely with density $%
f_{K}\left( x|y,\rho ,0\right) ,$ we deduce (by simple change of variables
in appropriate integral) that polynomials $\left\{ k_{n}\left( x\sqrt{1-q}|y%
\sqrt{1-q},\rho \right) \right\} _{n\geq 0}$ are orthogonal with respect to $%
f_{K}\left( x|y,\rho ,q\right) .$
\end{remark}

Hence in particular $f_{N}$ is a generalization of $N\left( 0,1\right) $
density, while $f_{CN}$ is a generalization of $N\left( \rho y,1-\rho
^{2}\right) $ density. It is also known see e.g. \cite{bms} that $%
f_{CN}\left( x|y,\rho ,q\right) /f_{N}\left( x|q\right) \allowbreak $
follows Lancaster type expansion (see e.g. \cite{Lancaster}). Namely we have:

\begin{gather}
\prod_{k=0}^{\infty }\frac{(1-\rho ^{2}q^{k})}{(1-\rho
^{2}q^{2k})^{2}-(1-q)\rho q^{k}(1+\rho ^{2}q^{2k})xy+(1-q)\rho
^{2}(x^{2}+y^{2})q^{2k}}  \label{PM} \\
=\sum_{n=0}^{\infty }\frac{\rho ^{n}}{[n]_{q}!}H_{n}(x|q)H_{n}(y|q),  \notag
\end{gather}%
converges uniformly and defines the Poisson--Mehler kernel. It is an almost
obvious generalization of (\ref{MHPF}) and (\ref{PMq=0}). We will prove and
generalize it by the expansion idea of this paper in the next section.

\section{Auxiliary results\label{auxi}}

In this section we are going either to recall or to calculate connection
coefficients of one family of orthogonal polynomials with respect to the
others. First we will recall known results, exposing some of the families of
connection coefficients. To do this let us introduce one more family of
polynomials $\left\{ B_{n}\left( x|q\right) \right\} _{n\geq 0}$ that are
orthogonal but with respect to some complex measure. They play an auxiliary
role and satisfy the following $3-$term recursive equation:%
\begin{equation}
B_{n+1}\left( y|q\right) \allowbreak =\allowbreak -q^{n}yB_{n}\left(
y|q\right) +q^{n-1}\left[ n\right] _{q}B_{n-1}\left( y|q\right) ;n\geq 0,
\label{_B}
\end{equation}%
with $B_{-1}\left( y|q\right) =0,$ $B_{0}\left( y|q\right) =1.$ Formula (16)
of \cite{bms} allows to express them through $q-$Hermite polynomials. 
\newline
Namely we have: $B_{n}(x|q)=\left\{ 
\begin{array}{ll}
i^{n}q^{n(n-2)/2}H_{n}(i\sqrt{q}\,x|\frac{1}{q}) & \text{for }q>0 \\ 
(-1)^{n(n-1)/2}|q|^{n(n-2)/2}H_{n}(-\sqrt{|q|}\,x|\frac{1}{q}) & \text{for }%
q<0%
\end{array}%
,\right. $ where $i\allowbreak =\allowbreak \sqrt{-1}.$ Obviously we have $%
B_{n}\left( x|0\right) \allowbreak =\allowbreak 0$ for $n>2$ and also one
can see that $B_{n}\left( x|1\right) \allowbreak =\allowbreak
i^{n}H_{n}\left( iy\right) ,$ $n\geq 0.$

The properties of families of polynomials $\left\{ H_{n}\right\} _{n\geq 0},$
$\left\{ P_{n}\right\} _{n\geq 0},$ $\left\{ R_{n}\right\} _{n\geq 0},$
including 'connection coefficient formulae' met in the literature, are
collected in the following Lemma

\begin{lemma}
\label{znane}i) $\forall n\geq 1:P_{n}\left( x|y,\rho ,q\right)
=\sum_{j=0}^{n}\QATOPD[ ] {n}{j}\rho ^{n-j}B_{n-j}\left( y|q\right)
H_{j}\left( x|q\right) ,$

ii) $\forall n>0:\sum_{j=0}^{n}\QATOPD[ ] {n}{j}B_{n-j}\left( x|q\right)
H_{j}\left( x|q\right) =0,$

iii) $\forall n\geq 0:H_{n}\left( x|q\right) =\sum_{j=0}^{n}\QATOPD[ ] {n}{j}%
\rho ^{n-j}H_{n-j}\left( y|q\right) P_{j}\left( x|y,\rho ,q\right) ,$

iv) $\forall n\geq 0:$ $U_{n}\left( x\sqrt{1-q}/2\right)
=\sum_{j=0}^{\left\lfloor n/2\right\rfloor }\left( -1\right)
^{j}(1-q)^{n/2-j}q^{j\left( j+1\right) /2}\QATOPD[ ] {n-j}{j}%
_{q}H_{n-2j}\left( x|q\right) $ and $H_{n}\left( y|q\right)
=\sum_{k=0}^{\left\lfloor n/2\right\rfloor }(1-q)^{-n/2}q^{k}\left( \QATOPD[
] {n}{k}_{q}-q^{n-2k+1}\QATOPD[ ] {n}{k-1}_{q}\right) U_{n-2k}\left( y\sqrt{%
1-q}/2\right) ,$

v) $\forall n\geq 1,\left\vert \beta \right\vert ,\left\vert \gamma
\right\vert <1:R_{n}\left( x|\gamma ,q\right) \allowbreak =\allowbreak
\sum_{k=0}^{\left\lfloor n/2\right\rfloor }\beta ^{k}\frac{\left[ n\right]
_{q}!\left( \gamma /\beta \right) _{k}\left( \gamma \right) _{n-k}(1-\beta
q^{n-2k})}{\left[ k\right] _{q}!\left[ n-2k\right] _{q}!\left( \beta
q\right) _{n-k}(1-\beta )}R_{n-2k}\left( x|\beta ,q\right) ,$ \newline
in particular: $R_{n}\left( x|\gamma ,q\right) \allowbreak =\allowbreak
\sum_{k=0}^{\left\lfloor n/2\right\rfloor }(-1)^{k}\gamma ^{k}q^{k(k-1)/2}%
\frac{\left[ n\right] _{q}!\left( \gamma \right) _{n-k}}{\left[ k\right]
_{q}!\left[ n-2k\right] _{q}!}H_{n-2k}\left( x|q\right) $ \newline
and $H_{n}\left( x|q\right) \allowbreak =\allowbreak
\sum_{k=0}^{\left\lfloor n/2\right\rfloor }\beta ^{k}\frac{\left[ n\right]
_{q}!\left( 1-\beta q^{n-2k}\right) }{(1-\beta )\left[ k\right] _{q}!\left[
n-2k\right] _{q}!\left( \beta q\right) _{n-k}}R_{n-2k}\left( x|\beta
,q\right) .$
\end{lemma}

\begin{proof}
Formulae given in assertions i) and ii) are given in Remark 1 following
Theorem 1 in \cite{bms}. iii) We start with formula (4.7) in \cite{IRS99}
that gives connection coefficients of $h_{n}$ with respect to $p_{n}.$ Then
we pass to polynomials $H_{n}$ \& $P_{n}$ using formulae $h_{n}\left(
x|q\right) \allowbreak =\allowbreak \left( 1-q\right) ^{n/2}H_{n}\left( 
\frac{2x}{\sqrt{1-q}}|q\right) ,$ $n\geq 1$ and $p_{n}(x|a,b,q)\allowbreak
=\allowbreak \left( 1-q\right) ^{n/2}P_{n}\left( \frac{2x}{\sqrt{1-q}}|\frac{%
2a}{\sqrt{\left( 1-q\right) b}},\sqrt{b},q\right) .$ By the way notice that
this formula can be easily derived from assertions i) and ii) by standard
change of order of summation. iv) Follows 'change of base' formula in
continuous $q-$Hermite polynomials $(i.e.$ polynomials $h_{n})$ in e.g. \cite%
{IS03}, \cite{Bress80} or \cite{GIS99} (formula 7.2) that states that 
\begin{equation*}
h_{n}\left( x|p\right) =\sum_{k=0}^{\left\lfloor n/2\right\rfloor
}c_{n,n-2k}\left( p,q\right) h_{n-2k}\left( x|q\right)
\end{equation*}%
where 
\begin{eqnarray*}
c_{n,n-2k}\left( p,q\right) &=&\sum_{j=0}^{k}\left( -1\right)
^{j}p^{k-j}q^{j\left( j+1\right) /2}\QATOPD[ ] {n-2k+j}{j}_{q}\allowbreak
\times \allowbreak \\
&&(\QATOPD[ ] {n}{k-j}_{p}-p^{n-2k+2j+1}\QATOPD[ ] {n}{k-j-1}_{p})
\end{eqnarray*}%
again expressed for polynomials $h_{n}$, next one observes that $h_{n}\left(
x|0\right) \allowbreak =\allowbreak U_{n}\left( x\right) ,$ $\QATOPD[ ] {n}{k%
}_{0}\allowbreak =\allowbreak 1$ for $n\geq 0,$ $k\allowbreak =\allowbreak
0,\ldots ,n$ hence we have 
\begin{equation*}
c_{n,n-2k}\left( 0,q\right) \allowbreak =\allowbreak \left( -1\right)
^{k}q^{k\left( k+1\right) /2}\QATOPD[ ] {n-k}{k}_{q}
\end{equation*}
and consequently 
\begin{equation*}
U_{n}\left( x\right) =\sum_{k=0}^{\left\lfloor n/2\right\rfloor }\left(
-1\right) ^{k}q^{k\left( k+1\right) /2}\QATOPD[ ] {n-k}{k}_{q}h_{n-2k}\left(
x|q\right) ,
\end{equation*}
similarly we get 
\begin{equation*}
c_{n,n-2k}\left( q,0\right) \allowbreak =\allowbreak q^{k}(\QATOPD[ ] {n}{k}%
_{q}\allowbreak -\allowbreak q^{n-2k+1}\QATOPD[ ] {n}{k-1}_{q})
\end{equation*}
and consequently 
\begin{equation*}
h_{n}\left( x|q\right) \allowbreak =\allowbreak \sum_{k=0}^{\left\lfloor
n/2\right\rfloor }q^{k}(\QATOPD[ ] {n}{k}_{q}\allowbreak -\allowbreak
q^{n-2k+1}\QATOPD[ ] {n}{k-1}_{q})U_{n-2k}\left( x\right) .
\end{equation*}
Now it remains to return to polynomials $H_{n}$ . v) It is in fact the
celebrated connection coefficient formula for the Rogers polynomials which
was expressed in term of the polynomials $C_{n}$ (see 13.3.5 of \cite{Is2005}%
). Other formulae in this assertions are in fact applications of the first
formula with $\beta \allowbreak =\allowbreak 0$ in the first case and $%
\gamma \allowbreak =\allowbreak 0$ in the second and using the fact that $%
R_{n}\left( x|0,q\right) \allowbreak =\allowbreak H_{n}\left( x|q\right) .$
\end{proof}

We have an important proposition generalizing assertion ii) of the Lemma
above. We will use it in the proof of the Lemma \ref{connections} below.

\begin{lemma}
\label{connections}$\forall n\geq 0:\allowbreak $

$i)$ 
\begin{equation}
U_{n}\left( x\sqrt{1-q}/2\right) =\sum_{k=0}^{n}D_{k,n}\left( y,\rho
,q\right) P_{k}\left( x|y,\rho ,q\right) ,  \label{expUnap}
\end{equation}%
where%
\begin{eqnarray*}
D_{k,n}\left( y,\rho ,q\right) &=&\sum_{j=0}^{\left\lfloor
(n-k)/2\right\rfloor }\left( -1\right) ^{j}\left( 1-q\right)
^{n/2-j}q^{j\left( j+1\right) /2}\QATOPD[ ] {n-j}{n-k-j}\times \\
&&\QATOPD[ ] {n-k-j}{n-k-2j}\rho ^{n-k-2j}H_{n-k-2j}\left( y|q\right) .
\end{eqnarray*}

$ii)$ 
\begin{equation}
k_{n}\left( x\sqrt{1-q}|y\sqrt{1-q},\rho \right) \allowbreak
=\sum_{k=0}^{n}C_{k,n}\left( y,\rho ,q\right) P_{k}\left( x|y,\rho ,q\right)
,  \label{expansion}
\end{equation}%
where

\begin{eqnarray*}
C_{k,n}\left( y,\rho ,q\right) &=&\sum_{j=0}^{\left\lfloor
(n-k)/2\right\rfloor }\left( -1\right) ^{j}\left( 1-q\right)
^{n/2-j}q^{n-k+j\left( j-3\right) /2}\QATOPD[ ] {n-1-j}{n-k-2j}_{q}\times \\
&&\left( \QATOPD[ ] {j+k}{k}_{q}-\rho ^{2}q^{k}\QATOPD[ ] {j+k-1}{k}%
_{q}\right) \rho ^{n-k-2j}H_{n-k-2j}\left( y|q\right) .
\end{eqnarray*}
\end{lemma}

\begin{remark}
Notice that 
\begin{equation*}
D_{k,n}\left( y,\rho ,q\right) \allowbreak \left( \rho ^{2}\right) _{k}\left[
k\right] _{q}!\allowbreak =\allowbreak \int_{-2/\sqrt{1-q}}^{2/\sqrt{1-q}%
}U_{n}\left( x\sqrt{1-q}/2\right) P_{k}\left( x|y,\rho ,q\right)
f_{CN}\left( x|y,\rho ,q\right) dx
\end{equation*}%
$\allowbreak $

and 
\begin{equation*}
C_{k,n}\left( y,\rho ,q\right) \allowbreak \left( \rho ^{2}\right) _{k}\left[
k\right] _{q}!\allowbreak =\newline
\allowbreak \int_{-2/\sqrt{1-q}}^{2/\sqrt{1-q}}P_{n}\left( x\sqrt{1-q}|y%
\sqrt{1-q},\rho ,0\right) P_{k}\left( x|y,\rho ,q\right) f_{CN}\left(
x|y,\rho ,q\right) dx.
\end{equation*}
\end{remark}

Let us define the following quantity: 
\begin{equation*}
\lbrack 2k-1]_{q}!!\allowbreak =\allowbreak \left\{ 
\begin{array}{ccc}
1 & if & k=0 \\ 
\prod_{i=1}^{k}[2i-1]_{q} & if & k\geq 1%
\end{array}%
\right. .
\end{equation*}

We have also some interesting corollaries based on the following easy,
elementary observations contained in the Remark below. It is following
simple induction applied to formulae (\ref{_0}), (\ref{H}), Proposition \ref%
{uwaga} 3. , (\ref{_B}), and (\ref{R}).

\begin{remark}
\label{zera}i) $U_{n}\left( 0\right) \allowbreak =\allowbreak \left\{ 
\begin{array}{ccc}
0 & if & n=2k-1 \\ 
(-1)^{k} & if & n=2k%
\end{array}%
\right. $ , $k=1,2,\ldots $

ii) $U_{n}\left( 1\right) \allowbreak =\allowbreak \left( -1\right)
^{n}U_{n}\left( -1\right) \allowbreak =\allowbreak \left( n+1\right) ,$

iii) $U_{n}\left( \frac{1}{2}\right) \allowbreak =\allowbreak \left(
-1\right) ^{3\left\lfloor (n+2)/3\right\rfloor }\left( (n+1-3\left\lfloor
(n+2)/3\right\rfloor \right) ,$ \newline
iv) $H_{n}\left( 0|q\right) \allowbreak =\allowbreak \left\{ 
\begin{array}{ccc}
0 & if & n=2k-1 \\ 
(-1)^{k}\left[ 2k-1\right] _{q}!! & if & n=2k%
\end{array}%
\right. ,$ $k\allowbreak =\allowbreak 1,2\ldots $ ,

$H_{n}\left( \frac{2}{\sqrt{1-q}}\right) \allowbreak =\allowbreak \frac{%
W_{n}\left( q\right) }{\left( 1-q\right) ^{n/2}},$ where $W_{n}\left(
q\right) $ by \ref{maksima} and $n\geq 1,$ \newline
v) $k_{n}\left( 0|y,\rho \right) \allowbreak =\allowbreak \left\{ 
\begin{array}{ccc}
(-1)^{k}\left( 1-\rho ^{2}\right) & if & n=2k \\ 
(-1)^{k-1}\rho y & if & n=2k-1%
\end{array}%
\right. ,$ $k=1,2,\ldots $ \newline
$k_{n}\left( 1|y,\rho \right) \allowbreak =\allowbreak \left\{ 
\begin{array}{ccc}
(-1)^{k}(1-\rho ^{2}) & if & n=3k \\ 
(-1)^{k-1}(-\rho y+\rho ^{2}) & if & n=3k-1 \\ 
\left( -1\right) ^{k-1}(1-\rho y) & if & n=3k-2%
\end{array}%
\right. ,$ $k=1,2,\ldots $

vi) $B_{n}\left( 0|q\right) \allowbreak =\allowbreak \left\{ 
\begin{array}{ccc}
0 & if & n=2k-1 \\ 
q^{k(k-1)}\left[ 2k-1\right] _{q}!! & if & n=2k%
\end{array}%
\right. ,$ $k\allowbreak =\allowbreak 1,2,\ldots $\newline
vii) $R_{n}\left( 0,\beta ,q\right) \allowbreak =\allowbreak \left\{ 
\begin{array}{ccc}
0 & if & n=2k-1 \\ 
(-1)^{k}\left( \beta ^{2};q^{2}\right) _{k}\left[ 2k-1\right] _{q}!! & if & 
n=2k%
\end{array}%
.\right. $
\end{remark}

\begin{corollary}
$\forall \rho ,q\in (-1,1);n\geq 1:$

i) 
\begin{equation*}
1-q^{n(n+1)/2}\allowbreak =\allowbreak
\sum_{j=0}^{n-1}(1-q)^{n-j}q^{j(j+1)/2}\QATOPD[ ] {2n-j}{j}%
_{q}[2n-2j-1]_{q}!!
\end{equation*}

ii) 
\begin{equation*}
P_{n}\left( 0|y,\rho ,q\right) =\sum_{j=0}^{\left\lfloor n/2\right\rfloor }%
\QATOPD[ ] {n}{2j}_{q}\left( -1\right) ^{j}\rho ^{n-2j}B_{n-2j}\left(
y|q\right) \left[ 2j-1\right] _{q}!!\allowbreak
\end{equation*}
\end{corollary}

\begin{proof}
i) We put $x\allowbreak =\allowbreak 0$ in Lemma \ref{znane} iv), use
assertion Remark \ref{zera} iv), substitute $n\longrightarrow 2n,$ perform
necessary simplifications, we get including fact that: $(1-q)^{k}\left[ 2k-1%
\right] _{q}!!\allowbreak =\allowbreak (q|q^{2})_{k-1}$ and $H_{0}\left(
0|q\right) \allowbreak =\allowbreak 1$ which leads to conclusion that the
summand for $j\allowbreak =\allowbreak n~$is equal to $\allowbreak
q^{n(n+1)/2}$.

ii) We put $x=0$ and apply Lemma \ref{znane} iii) and then use Remark \ref%
{zera} iv).
\end{proof}

\section{Expansions\label{exp}}

In this section we are going to apply the general idea of expansion
presented in Section \ref{idea}, use results of Section \ref{auxi} and
obtain expansions of some presented above densities in terms of the others.
Since there will be many such expansions to formulate all of them in one
theorem would lead to clumsy and unclear statement. Instead we divide this
section unto many subsections entitled by the names of the densities that
will be discussed in its body.

\subsection{$f_{N}$ and $f_{U}$}

Using assertion Lemma \ref{znane} iv) we deduce that coefficients $\gamma
_{0,n}$ in expanding $f_{N}$ is given by $\gamma _{0,n}\allowbreak
=\allowbreak \left\{ 
\begin{array}{ccc}
0 & if & n=2k+1, \\ 
\left( -1\right) ^{k}q^{k\left( k+1\right) /2} & if & n=2k,%
\end{array}%
\right. $ $k\allowbreak =\allowbreak 0,1,\ldots $ and we end up with an
expansion 
\begin{equation}
f_{N}\left( x|q\right) \allowbreak =\allowbreak f_{U}\left( x|q\right)
\sum_{k=0}^{\infty }\left( -1\right) ^{k}q^{k\left( k+1\right)
/2}U_{2k}\left( x\sqrt{1-q}/2\right) ,  \label{NnaU}
\end{equation}%
which was obtained and discussed in \cite{szab2} with a help of so called
"triple product identity". This formula was recently successfully applied to
prove "free infinite divisibility" of the $q$-Normal (defined above)
distribution. For details see \cite{ABBL}.

Using another assertion of Lemma \ref{znane} iv) we get the reciprocal of
the above expansion. Namely we have 
\begin{equation*}
\gamma _{0,n}\allowbreak =\allowbreak \left\{ 
\begin{array}{ccc}
0 & if & n=2k+1, \\ 
\left( 1-q\right) ^{-k}q^{k}\left( \QATOPD[ ] {2k}{k}_{q}-q\QATOPD[ ] {2k}{%
k-1}_{q}\right) & if & n=2k,%
\end{array}%
\right. .
\end{equation*}%
Notice that $\left( 1-q\right) ^{-k}q^{k}\left( \QATOPD[ ] {2k}{k}_{q}-q%
\QATOPD[ ] {2k}{k-1}_{q}\right) /[2k]_{q}!\allowbreak =\allowbreak \frac{%
q^{k}\left( 1-q\right) ^{k+1}}{\left( q\right) _{k}\left( q\right) _{k+1}}.$
Since we have also (\ref{z H}), we get: 
\begin{equation}
f_{U}\left( x|q\right) \allowbreak =\allowbreak f_{N}\left( x|q\right)
\sum_{k=0}^{\infty }\frac{q^{k}\left( 1-q\right) ^{k+1}}{\left( q\right)
_{k}\left( q\right) _{k+1}}H_{2k}\left( x|q\right) .  \label{UnaN}
\end{equation}

As corollaries we get the following useful formulae that were exposed
already in \cite{szab2} and which are presented here for completeness: 
\begin{equation*}
\left( q\right) _{\infty }\prod_{k=1}^{\infty }\left(
(1+q^{k})^{2}-(1-q)x^{2}q^{k}\right) \allowbreak =\allowbreak
\sum_{k=0}^{\infty }\left( -1\right) ^{k}q^{k\left( k+1\right)
/2}U_{2k}\left( x\sqrt{1-q}/2\right) ,
\end{equation*}%
which reduces (after putting $x\allowbreak =\allowbreak 0)$ to well known 
\begin{equation*}
\left( q\right) _{\infty }\left( -q\right) _{\infty }^{2}\allowbreak
=\allowbreak \left( -q\right) _{\infty }\left( q^{2}|q^{2}\right) _{\infty
}\allowbreak =\allowbreak \sum_{k=0}^{\infty }q^{k\left( k+1\right) /2}
\end{equation*}%
which is a particular case of the 'triple product identity' or (after
putting $x^{2}(1-q)\allowbreak =4)\allowbreak $ to:%
\begin{equation*}
\left( q\right) _{\infty }^{3}\allowbreak =\allowbreak \sum_{k=0}^{\infty
}\left( -1\right) ^{k}\left( 2k+1\right) q^{k(k+1)/2}.
\end{equation*}%
Similarly analyzing (\ref{UnaN}) we get: 
\begin{equation*}
\prod_{k=1}^{\infty }\left( (1+q^{k})^{2}-(1-q)x^{2}q^{k}\right)
^{-1}\allowbreak =\allowbreak \sum_{k=0}^{\infty }\frac{q^{k}\left(
q^{k+1}\right) _{\infty }\left( 1-q\right) ^{k}}{\left( q^{2}\right) _{k}}%
H_{2k}\left( x|q\right) ,
\end{equation*}%
since $\left( q\right) _{\infty }/\left( q\right) _{k}\allowbreak
=\allowbreak \left( q^{k+1}\right) _{\infty }$ and $\left( q\right)
_{k+1}\allowbreak =\allowbreak \left( 1-q\right) \left( q^{2}\right) _{k}$
from which we get for example (by setting $x=0)$ identity 
\begin{equation*}
\frac{1}{\left( q\right) _{\infty }\left( -q\right) _{\infty }^{2}}%
\allowbreak =\allowbreak 1+\sum_{k=1}^{\infty }(-1)^{k}\frac{q^{k}\left(
1-q\right) ^{k}}{\left( q\right) _{k}\left( q^{2}\right) _{k}}\left[ 2k-1%
\right] _{q}!!,
\end{equation*}%
or (after inserting $x^{2}\left( 1-q\right) \allowbreak =\allowbreak 4$ and
applying Remark \ref{zera} iv)): 
\begin{equation*}
\left( q\right) _{\infty }^{-3}\allowbreak =\allowbreak \sum_{k=0}^{\infty }%
\frac{q^{k}W_{2k}\left( q\right) }{\left( q\right) _{k}\left( q^{2}\right)
_{k}}.
\end{equation*}

\subsection{$f_{N}$ and $f_{CN}$}

We use Lemma \ref{znane} i) we deduce that coefficients $\gamma _{0,n}$ in
expanding $f_{CN}$ are given by $\gamma _{0,n}\allowbreak =\allowbreak \rho
^{n}B_{n}\left( y|q\right) .$ Keeping in mind (\ref{z p}) we get%
\begin{equation}
f_{N}\left( x|q\right) \allowbreak =\allowbreak f_{CN}\left( x|y,\rho
,q\right) \sum_{n=0}^{\infty }\frac{\rho ^{n}}{\left( \rho ^{2}\right)
_{n}[n]_{q}!}B_{n}\left( y|q\right) P_{n}\left( x|y,\rho ,q\right) .
\label{NnaCN}
\end{equation}

We use Lemma \ref{znane} iii) we deduce that coefficients $\gamma _{0,n}$ in
expanding $f_{CN}$ is given by $\gamma _{0,n}\allowbreak =\allowbreak \rho
^{n}H_{n}\left( y|q\right) .$ Keeping in mind (\ref{z H}) we get:%
\begin{equation}
f_{CN}\left( x|y,\rho ,q\right) \allowbreak =\allowbreak f_{N}\left(
x|q\right) \sum_{n=0}^{\infty }\frac{\rho ^{n}}{[n]_{q}!}H_{n}\left(
y|q\right) H_{n}\left( x|q\right) .  \label{CNnaN}
\end{equation}%
Notice that (\ref{CNnaN}) it is in fact the famous Poisson--Mehler kernel of
the $q-$ Hermite polynomials, while (\ref{NnaCN}) is its reciprocal. Compare 
\cite{bressoud} for another proof of (\ref{CNnaN}). Notice that for every
fixed $m,$ $\sum_{n=0}^{m}\frac{\rho ^{n}}{\left( \rho ^{2}\right)
_{n}[n]_{q}!}B_{n}\left( y|q\right) P_{n}\left( x|y,\rho ,q\right) $ is not
a symmetric function of $x$ and $y,$ while when $m\allowbreak =\allowbreak
\infty $ it is!

As a corollary (after putting $y\allowbreak =\allowbreak x$ and then using
Remark \ref{fRifN}) we get the following interesting expansion 
\begin{equation}
\frac{\left( \rho ^{2}\right) _{\infty }}{\left( \rho \right) _{\infty
}^{2}\prod_{k=0}^{\infty }\left( \left( 1+\rho q^{k}\right) ^{2}-\left(
1-q\right) \rho x^{2}q^{k}\right) }=\sum_{n=0}^{\infty }\frac{\rho ^{n}}{%
\left[ n\right] _{q}!}H_{n}^{2}\left( x|q\right) ,  \label{H_kw1}
\end{equation}%
which reduces to the well known formula (see \cite{Is2005}, Exercise
12.3(b)) 
\begin{equation*}
\frac{\left( \rho ^{2}\right) _{\infty }}{\left( \rho \right) _{\infty }^{4}}%
\allowbreak =\allowbreak \sum_{n=0}^{\infty }\frac{\rho ^{n}}{\left(
q\right) _{n}}W_{n}^{2}\left( q\right) ,
\end{equation*}%
after inserting $x\allowbreak =\allowbreak 2/\sqrt{1-q}$ and applying \ref%
{maksima} with $W_{n}$ defined by (\ref{WandV}). Expansion (\ref{H_kw1})
after inserting $x\allowbreak =\allowbreak 0,$ can be reduced to: 
\begin{equation*}
\prod_{k=0}^{\infty }\frac{\left( 1-\rho ^{2}q^{2k+1}\right) }{(1-\rho
^{2}q^{2k})}\allowbreak =\allowbreak 1+\sum_{k=1}^{\infty }\rho
^{2k}\prod_{j=1}^{k}\frac{\left( 1-q^{2j-1}\right) }{\left( 1-q^{2j}\right) }%
,
\end{equation*}%
since as it can be easily noticed $\frac{\left( \left[ 2k-1\right]
_{q}!!\right) ^{2}}{\left[ 2k\right] _{q}!}\allowbreak =\allowbreak
\prod_{j=1}^{k}\frac{\left( 1-q^{2j-1}\right) }{\left( 1-q^{2j}\right) }$
and $\frac{\left( \rho ^{2}\right) _{\infty }}{\left( \rho \right) _{\infty
}^{2}\left( -\rho ^{2}\right) _{\infty }^{2}}\allowbreak =\allowbreak
\prod_{k=0}^{\infty }\frac{\left( 1-\rho ^{2}q^{2k+1}\right) }{(1-\rho
^{2}q^{2k})}.$

As far as convergence of series (\ref{NnaCN}) and (\ref{CNnaN}) is concerned
then we see that for $\left\vert \rho \right\vert ,\left\vert q\right\vert
<1 $ and $x,y$ $\in S_{q}$ function $g(x|y,\rho ,q)\allowbreak =\allowbreak
f_{CN}\left( x|y,\rho ,q\right) /f_{N}\left( x|q\right) \allowbreak
=\allowbreak \left( \rho _{2}\right) _{\infty }\prod_{k=0}^{\infty }\frac{1}{%
(1-\rho ^{2}q^{2k})^{2}-(1-q)\rho q^{k}(1+\rho ^{2}q^{2k})xy+(1-q)\rho
^{2}(x^{2}+y^{2})q^{2k}}$ both bounded and 'cut away from zero' hence its
square as well as reciprocal of this square are integrable on compact
interval $S_{q}.$ For exact bounds see \cite{Szab6} Proposition 1 vii).

\begin{remark}
Dividing both sides of (\ref{NnaCN}) and (\ref{CNnaN}) by $f_{N}\left(
x|q\right) ,$ letting $q\longrightarrow 1^{-}$ and keeping in mind that $%
B_{n}\left( x|1\right) \allowbreak =\allowbreak i^{n}H_{n}\left( ix\right) $
and that \newline
$P_{n}\left( x|y,\rho ,1\right) \allowbreak =\allowbreak \left( \sqrt{1-\rho
^{2}}\right) ^{n}H_{n}\left( \frac{(x-\rho y)}{\sqrt{1-\rho ^{2}}}\right) $%
we get:%
\begin{equation}
1/\sum_{n=0}^{\infty }\frac{\rho ^{n}}{n!}H_{n}\left( x\right) H_{n}\left(
y\right) \allowbreak =\allowbreak \sum_{n=0}^{\infty }\frac{\rho ^{n}i^{n}}{%
n!\left( 1-\rho ^{2}\right) ^{n/2}}H_{n}\left( ix\right) H_{n}\left( \frac{%
(x-\rho y)}{\sqrt{1-\rho ^{2}}}\right)  \label{odwracanie}
\end{equation}%
Here however situation is different. The series $\sum_{n=0}^{\infty }\frac{%
\rho ^{n}}{n!}H_{n}\left( x\right) H_{n}\left( y\right) $, as it is known,
is convergent for all $x,y\in \mathbb{R}$ and $\left\vert \rho \right\vert
<1,$ while the series (\ref{odwracanie}) only for $x,y\in \mathbb{R}$ and $%
\rho ^{2}<1/2$ since only then the function $f_{N}^{2}(x|q)/f_{CN}\left(
x|y,\rho ,q\right) \allowbreak =\allowbreak \exp (-\frac{(x-\rho y)^{2}}{%
2(1-\rho ^{2})}+x^{2})$ is integrable with respect to $x$ over whole $%
\mathbb{R}.$
\end{remark}

\subsection{ $f_{N}$ and $f_{R}$}

We use the last two statements of Lemma \ref{znane} v). We deduce that
coefficients $\gamma _{0,n}$ in expanding $f_{R}$ are given by 
\begin{equation*}
\gamma _{0,n}\allowbreak =\allowbreak \left\{ 
\begin{array}{ccc}
0 & if & n=2k+1, \\ 
\frac{\left[ 2k\right] _{q}!\beta ^{k}}{\left[ k\right] _{q}!\left( \beta
q\right) _{k}} & if & n=2k,%
\end{array}%
\right. ,
\end{equation*}%
$k\allowbreak \allowbreak =\allowbreak 0,1,\ldots .$ Keeping in mind (\ref{z
H}) we get:%
\begin{equation}
f_{R}\left( x|\beta ,q\right) \allowbreak =\allowbreak f_{N}\left(
x|q\right) \sum_{k=0}^{\infty }\frac{\beta ^{k}}{[k]_{q}!\left( \beta
q\right) _{k}}H_{2k}\left( x|q\right) .  \label{RnaN}
\end{equation}%
As a corollary let us take $\beta \allowbreak =\allowbreak \rho $ and use (%
\ref{fRnafN}) and compare it with (\ref{H_kw1}). We will get then for $%
\left\vert q\right\vert ,\left\vert \rho \right\vert <1,x^{2}\left(
1-q\right) \leq 2$:%
\begin{equation*}
\left( 1-\rho \right) \sum_{n=0}^{\infty }\frac{\rho ^{n}}{\left[ n\right]
_{q}!}H_{n}^{2}\left( x\right) \allowbreak =\allowbreak \sum_{n=0}^{\infty }%
\frac{\rho ^{n}}{\left[ n\right] _{q}!\left( \rho q\right) _{n}}H_{2n}\left(
x|q\right) .
\end{equation*}%
Next we use second assertion of v) of Lemma \ref{znane} and deduce that
coefficient $\gamma _{0,n}$ in expanding $f_{N}$ is by 
\begin{equation*}
\gamma _{0,n}\allowbreak =\allowbreak \left\{ 
\begin{array}{ccc}
0 & if & n=2k+1, \\ 
\left( -\gamma \right) ^{k}q^{k\left( k-1\right) /2}\frac{\left[ 2k\right]
_{q}!\left( \gamma \right) _{k}}{\left[ k\right] _{q}!} & if & n=2k,%
\end{array}%
\right. ,
\end{equation*}%
$k\allowbreak =\allowbreak 0,1,\ldots $ . We use also (\ref{z R}) and get 
\begin{equation}
f_{N}\left( x|q\right) \allowbreak =\allowbreak f_{R}\left( x|\gamma
,q\right) \sum_{k=0}^{\infty }\left( -\gamma \right) ^{k}q^{k\left(
k-1\right) /2}\frac{\left( \gamma \right) _{k}\left( 1-\gamma q^{2k}\right) 
}{\left( 1-\gamma \right) [k]_{q}!\left( \gamma ^{2}\right) _{2k}}%
R_{2k}\left( x|\gamma ,q\right) .  \label{NnaR}
\end{equation}%
Again we can deduce that one of the series (\ref{NnaR}) and (\ref{RnaN}) is
the reciprocal of the other.

\subsection{$f_{K}$ and $f_{CN}$}

Recall that the densities $f_{K}$ and $f_{CN}$ are given by (\ref{fK}) and (%
\ref{fCN}) respectively. We will be using Lemma \ref{connections} ii) ,
Remark \ref{kesten}, and the fact that for $n\geq 1:$ 
\begin{eqnarray*}
&&\int_{-2/\sqrt{1-q}}^{2/\sqrt{1-q}}f_{K}\left( \xi |y,\rho ,q\right)
k_{n}^{2}\left( \xi \sqrt{1-q}|y\sqrt{1-q},\rho \right) d\xi \\
&=&\frac{1}{\sqrt{1-q}}\int_{-2}^{2}f_{K}\left( x/\sqrt{1-q}|y/\sqrt{1-q}%
,\rho ,0\right) k_{n}^{2}\left( x|y,\rho \right) dx=\frac{\left( 1-\rho
^{2}\right) }{\sqrt{1-q}}.
\end{eqnarray*}%
Beside notice that $C_{0,1}\left( y,\rho ,q\right) \allowbreak =\allowbreak
1.$ Hence $\beta _{1}\left( y,\rho ,q\right) \allowbreak =\allowbreak 0.$
Consequently we get $\forall x\in <\frac{-2}{\sqrt{1-q}},\frac{2}{\sqrt{1-q}}%
>;\allowbreak y\in <\frac{-2}{\sqrt{1-q}},\frac{2}{\sqrt{1-q}}>;\allowbreak
0<\left\vert \rho \right\vert <1;\allowbreak q\in (-1,1)$ 
\begin{equation*}
f_{CN}\left( x|y,\rho ,q\right) \allowbreak =\allowbreak f_{K}\left( x|\rho
,q\right) (1+\allowbreak \sum_{n=2}^{\infty }\beta _{n}\left( y,\rho
,q\right) k_{n}(x\sqrt{1-q}|y\sqrt{1-q},\rho ))),
\end{equation*}%
where $\beta _{k}\left( y,\rho ,q\right) \allowbreak =\allowbreak
\sum_{j=1}^{\left\lfloor k/2\right\rfloor }(-1)^{j}\left( 1-q\right)
^{k/2-j}q^{k+j(j-3)/2}\QATOPD[ ] {k-1-j}{k-2j}\rho ^{k-2j}H_{k-2j}\left(
y|q\right) .$

\subsection{$f_{U}$ and $f_{CN}$}

Using Lemma \ref{connections} i) and calculating in the similar way we get: $%
\forall x\in <\frac{-2}{\sqrt{1-q}},\frac{2}{\sqrt{1-q}}>;\allowbreak y\in <%
\frac{-2}{\sqrt{1-q}},\frac{2}{\sqrt{1-q}}>;\allowbreak 0<\left\vert \rho
\right\vert <1;\allowbreak q\in (-1,1),$ 
\begin{equation}
f_{CN}\left( x|y,\rho ,q\right) \allowbreak \allowbreak =\allowbreak
f_{U}\left( x|q\right) (1+\sum_{k=1}^{\infty }\gamma _{k}\left( y,\rho
,q\right) U_{k}\left( x\sqrt{1-q}/2\right) ),  \label{fCNnafU}
\end{equation}%
with $\gamma _{k}\left( y,\rho ,q\right) \allowbreak =\allowbreak
\sum_{j=0}^{\left\lfloor k/2\right\rfloor }\left( -1\right) ^{j}\left(
1-q\right) ^{k/2-j}\allowbreak \times \allowbreak q^{j\left( j+1\right) /2}%
\QATOPD[ ] {k-j}{k-2j}_{q}\rho ^{k-2j}H_{k-2j}\left( y|q\right) .$

\begin{corollary}
\begin{eqnarray*}
&&\left( q^{3};q^{3}\right) _{\infty }\sum_{k=0}^{\infty }\frac{\left(
1-q\right) ^{k/2}\rho ^{k}}{\left( q\right) _{k}}H_{k}\left( y|q\right) \eta
_{k}\left( q\right) \\
&=&\frac{\left( \rho ^{2}\right) _{\infty }\left( q^{3};q^{3}\right)
_{\infty }}{\prod_{k=0}^{\infty }\left( 1+\rho ^{2}q^{2k}+\rho ^{4}q^{4k}-%
\sqrt{1-q}\rho yq^{k}(1+\rho ^{2}q^{2k}\right) +(1-q)\rho ^{2}y^{2}q^{2k})}
\\
&=&1+\sum_{k=1}^{\infty }\left( -1\right) ^{3k}\left( \gamma _{3k}\left(
y,\rho ,q\right) +\gamma _{3k+1}\left( y,\rho ,q\right) \right) ,
\end{eqnarray*}%
where $\left\{ \eta _{k}\left( q\right) \right\} _{k\geq -1}$ are given
recursively $\eta _{-1}\left( q\right) \allowbreak =\allowbreak 0,\eta
_{0}\left( q\right) \allowbreak =\allowbreak 1,$ $\eta _{k+1}\left( q\right)
\allowbreak =\allowbreak \eta _{k}\left( q\right) \allowbreak -\allowbreak
\left( 1-q^{k}\right) \eta _{k-1}\left( q\right) ,$ $k\geq 0.$
\end{corollary}

\begin{proof}
$\allowbreak $We insert $x\allowbreak =\allowbreak 1/\sqrt{1-q}$ in (\ref%
{fCNnafU}) and use Remark \ref{zera} iii) which simplifies to simple rule $%
U_{3m+2}\left( 1/2\right) \allowbreak =\allowbreak 0,\allowbreak
U_{3m}\left( 1/2\right) \allowbreak =\allowbreak U_{3m+1}\left( 1/2\right)
\allowbreak \left( -1\right) ^{3m}$. Then we insert $x=1/\sqrt{1-q}$ in (\ref%
{fCN}) and (\ref{fU}) and use the fact that $(1-\rho
^{2}q^{2k})^{2}\allowbreak +\allowbreak \rho ^{2}q^{2k}\allowbreak
=\allowbreak 1+\rho ^{2}q^{2k}+\rho ^{4}q^{4k}.$ On the way we also use (\ref%
{PM}), identity $\left( q\right) _{\infty }\prod_{k=1}^{\infty }\left(
1+q^{k}+q^{2k}\right) \allowbreak =\allowbreak \prod_{k=1}^{\infty }\left(
1-q^{3k}\right) \allowbreak =\allowbreak \left( q^{3};q^{3}\right) _{\infty
} $ , the fact $\left( 1-q\right) ^{k/2}H_{k}\left( \frac{1}{\sqrt{1-q}}%
\right) \allowbreak =\allowbreak h_{k}\left( 1/2\right) $ and the fact the
continuous $q-$Hermite polynomials $h_{n}\left( x|q\right) $ satisfy
relationship: $h_{n+1}\left( x|q\right) \allowbreak =\allowbreak
2xh_{n}\left( x|q\right) \allowbreak -\allowbreak \left( 1-q^{n}\right)
h_{n-1}\left( x|q\right) .$
\end{proof}

\section{Proofs\label{dowody}}

Let us start this section with very brief recollection of basic facts
concerning orthogonal polynomials.

\begin{enumerate}
\item If $\left\{ D_{n}\left( x\right) \right\} _{n\geq 0}$ is a sequence of
polynomials with respect to certain signed measure, then $\left\{ \eta
_{n}D_{n}\left( x\right) \right\} _{n\geq 0},$ for any nonzero sequence of
reals $\left\{ \eta _{n}\right\} $ has the same property. Thus we can
consider only monic sequences of orthogonal polynomials.

\item Every monic sequence of orthogonal polynomials say $\left\{
D_{n}\left( x\right) \right\} _{n\geq 0}$ satisfies the so called three term
recurrence (3TR) that is there exist two sequences of reals $\left\{ \alpha
_{n}\right\} _{n\geq 0}$ and $\left\{ \beta _{n}\right\} _{n\geq 0}$ such
that for every $n\geq 0$ we have%
\begin{equation*}
xD_{n}\left( x\right) \allowbreak =\allowbreak D_{n+1}\left( x\right)
+\alpha _{n}D_{n}\left( x\right) +\beta _{n}D_{n-1}\left( x\right) ,
\end{equation*}%
with $D_{-1}\left( x\right) =0,D_{0}\left( x\right) =1.$

\item More over when we have a sequence of monic polynomials that satisfies
some 3TR with given sequences $\left\{ \alpha _{n}\right\} $ and $\left\{
\beta _{n}\right\} $ then there exists at least one signed measure such that
these polynomials are orthogonal with respect this measure. This statement
functions in the literature as "Favard's Theorem".

\item If $n\geq 0$ we have $\beta _{n}>0$ then this signed measure is a
positive measure.

\item There exists more subtle conditions imposed on sequences $\left\{
\alpha _{n}\right\} $ and $\left\{ \beta _{n}\right\} $ that guarantee that
the orthogonalizing measure is unique or that it has the density.
\end{enumerate}

For details see \cite{Akhiezer}, \cite{Is2005} or \cite{MAN}.

\begin{proof}[Proof of the Proposition \protect\ref{iloraz}]
i) Notice that $\phi _{n}$ is a monic polynomial of degree $n,$ for $n\geq
1. $ Now let us calculate $\int_{\mathbb{R}}\phi _{n}\left( x\right) B\left(
x\right) dx.$ We have :%
\begin{eqnarray*}
\int_{\mathbb{R}}\phi _{n}\left( x\right) B\left( x\right) dx
&=&\sum_{i=0}^{n}\sum_{j=0}^{N}f_{n-i}\frac{w_{j}}{\hat{a}_{j}}\int_{\mathbb{%
R}}a_{i}\left( x\right) a_{j}\left( x\right) A\left( x\right) dx \\
&=&\sum_{i=0}^{n}f_{n-i}w_{i}\allowbreak =\allowbreak 0
\end{eqnarray*}%
for $n\geq 1$ . Conversely, let us consider polynomial defined by $%
p_{n}\left( x\right) \allowbreak =\allowbreak \sum_{i=0}^{n}w_{n-i}\phi
_{i}\left( x\right) .$ \newline
We have $p_{n}(x)\allowbreak =\allowbreak
\sum_{i=0}^{n}w_{n-i}\sum_{j=0}^{i}f_{i-j}a_{j}\left( x\right) \allowbreak
=\allowbreak \sum_{j=0}^{n}a_{j}\left( x\right)
\sum_{i=j}^{n}w_{n-i}f_{i-j}\allowbreak =$\newline
$\allowbreak \sum_{j=0}^{n}a_{j}\left( x\right)
\sum_{k=0}^{n-j}w_{n-j-k}f_{k}\allowbreak =\allowbreak
\sum_{j=0}^{n}a_{j}\left( x\right) \sum_{s=0}^{n-j}w_{s}f_{n-j-s}\allowbreak
=\allowbreak a_{n}\left( x\right) .$

ii) Let $i\leq N.$ Keeping in mind representation of $W\left( x\right) $ and
orthogonality of polynomials $a_{i}\left( x\right) $ with respect to the
measure $\alpha $ we get. 
\begin{equation*}
\int_{\mathbb{R}}a_{i}\left( x\right) B\left( x\right) dx\allowbreak
=\allowbreak \int_{\mathbb{R}}a_{i}\left( x\right) W\left( x\right) A\left(
x\right) dx\allowbreak =\allowbreak w_{i}.
\end{equation*}%
Similarly if $i>N$ we get zero by the orthogonality of $\left\{
a_{i}\right\} _{i\geq 0}$with respect to $A\left( x\right) $.

iii) Let us define coefficients $c_{n,i}$ by the following expansion:%
\begin{equation*}
a_{n}\left( x\right) \allowbreak =\allowbreak
\sum_{i=0}^{n}c_{n,i}b_{i}\left( x\right) ,
\end{equation*}%
The fact that $\left\{ a_{n}\right\} $ and $\left\{ b_{n}\right\} $ are
monic implies that $\forall n\geq 0:c_{n,n}\allowbreak =\allowbreak 1.$ ii)
implies that $c_{i,0}\allowbreak =\allowbreak w_{i},$ $i\leq
n;c_{n,0}\allowbreak =\allowbreak 0$ for $n\geq N+1.$ Besides we have the
following relationships between coefficients $c_{n,i}$ that is implied by $%
3- $terms recurrences satisfied by families $\left\{ a_{i}\right\} $ and $%
\left\{ b_{i}\right\} .$ On one hand we have $xa_{n}\left( x\right)
\allowbreak =\allowbreak a_{n+1}\left( x\right) \allowbreak +\allowbreak
\alpha _{n}a_{n}\left( x\right) \allowbreak +\allowbreak \hat{\alpha}%
_{n}a_{n-1}\left( x\right) \allowbreak =\allowbreak b_{n+1}\left( x\right)
\allowbreak +\allowbreak (\alpha _{n}+c_{n+1,n})b_{n}\left( x\right)
\allowbreak +\allowbreak \sum_{i=0}^{n-1}\left( c_{n+1,i}\allowbreak
+\allowbreak \alpha _{n}c_{n,i}\allowbreak +\allowbreak \hat{\alpha}%
_{n}c_{n-1,i}\right) b_{i}\left( x\right) $ on the other $xa_{n}\left(
x\right) \allowbreak =\allowbreak \sum_{i=0}^{n}c_{n,i}(b_{i+1}\left(
x\right) \allowbreak +\allowbreak \beta _{i}b_{i}\left( x\right) \allowbreak
+\allowbreak \hat{\beta}_{i}b_{i-1}\left( x\right) )\allowbreak =\allowbreak
b_{n+1}\left( x\right) \allowbreak +\allowbreak (c_{n,n-1}\allowbreak
+\allowbreak \beta _{n})b_{n}\left( x\right) \allowbreak +\allowbreak
\sum_{i=1}^{n-1}(c_{n,i-1}\allowbreak +\allowbreak \beta
_{i}c_{n,i}\allowbreak +\allowbreak \hat{\beta}_{i}c_{n,i+1})b_{i}\left(
x\right) \allowbreak +\allowbreak \beta _{0}c_{n,0}\allowbreak +\allowbreak 
\hat{\beta}_{1}c_{n,1}.$ Equating these two sides we get:%
\begin{gather*}
\alpha _{n}+c_{n+1,n}\allowbreak =\allowbreak c_{n,n-1}+\beta _{n}, \\
\forall 1\leq i\leq n-1:c_{n+1,i}\allowbreak +\allowbreak \alpha
_{n}c_{n,i}\allowbreak +\allowbreak \hat{\alpha}_{n}c_{n-1,i}=c_{n,i-1}%
\allowbreak +\allowbreak \beta _{i}c_{n,i}\allowbreak +\allowbreak \hat{\beta%
}_{i}c_{n,i+1}, \\
c_{n+1,0}\allowbreak +\allowbreak \alpha _{n}c_{n,0}\allowbreak +\allowbreak 
\hat{\alpha}_{n}c_{n-1,0}\allowbreak =\allowbreak \beta
_{0}c_{n,0}\allowbreak +\allowbreak \hat{\beta}_{1}c_{n,1}.
\end{gather*}%
From the last of these equations we deduce that $c_{n,1}=0$ for $n\geq N+2.$
Similarly by considering equation 
\begin{equation*}
c_{n+1,1}\allowbreak +\allowbreak \alpha _{n}c_{n,1}\allowbreak +\allowbreak 
\hat{\alpha}_{n}c_{n-1,1}=c_{n,0}\allowbreak +\allowbreak \beta
_{i}c_{n,1}\allowbreak +\allowbreak \hat{\beta}_{i}c_{n,2}
\end{equation*}%
we deduce that $c_{n,2}\allowbreak =\allowbreak 0$ for $n\geq N+3$ and so
on. We see that then $c_{n,i}\allowbreak =\allowbreak 0$ for $n\geq
N\allowbreak +\allowbreak i\allowbreak +\allowbreak 1.$ In particular it
means that $c_{n,n-j}\allowbreak =\allowbreak 0$ for $j\geq N+1$ .
\end{proof}

\begin{proof}
of Lemma \ref{connections}. i) We will argue straightforwardly using Lemma %
\ref{znane} $i)$ and $ii)$ and then comparing it with assertion iv) of the
same Lemma.

We have \newline
$\sum_{k=0}^{n}D_{k,n}\left( y,\rho ,q\right) P_{k}\left( x|y,\rho ,q\right)
\allowbreak =\allowbreak \sum_{k=0}^{n}D_{k,n}\left( y,\rho ,q\right)
\sum_{i=0}^{k}\QATOPD[ ] {k}{i}_{q}\rho ^{k-i}B_{k-i}\left( y|q\right)
H_{i}\left( x|q\right) \allowbreak =\allowbreak $\newline
$\sum_{i=0}^{n}H_{i}\left( x|q\right) \sum_{k=i}^{n}\QATOPD[ ] {k}{i}%
_{q}D_{k,n}\left( y,\rho ,q\right) B_{k-i}\left( y|q\right) .$ Let us denote 
\newline
$G_{i,n}\left( y,\rho ,q\right) \allowbreak =\allowbreak \sum_{k=i}^{n}%
\QATOPD[ ] {k}{i}_{q}D_{k,n}\left( y,\rho ,q\right) B_{k-i}\left( y|q\right)
.$ We have using formula for $D_{k,n}\left( y,\rho ,q\right) .$ 
\begin{eqnarray*}
G_{i,n}(y,\rho ,q)\allowbreak &=&\allowbreak \sum_{k=i}^{n}\QATOPD[ ] {k}{i}%
_{q}\rho ^{k-i}B_{k-i}\left( y|q\right) \allowbreak \times \allowbreak
D_{k,n}\left( y,\rho ,q\right) \\
&=&\sum_{k=i}^{n}\QATOPD[ ] {k}{i}_{q}\rho ^{k-i}B_{k-i}\left( y|q\right)
\sum_{j=0}^{\left\lfloor (n-k)/2\right\rfloor }\left( -1\right) ^{j}\left(
1-q\right) ^{n/2-j}q^{j\left( j+1\right) /2}\QATOPD[ ] {n-j}{n-k-j}%
\allowbreak \\
&&\times \allowbreak \QATOPD[ ] {n-k-j}{n-k-2j}\rho
^{n-k-2j}H_{n-k-2j}\left( y|q\right) \\
&=&\sum_{j=0}^{\left\lfloor (n-i)/2\right\rfloor }\left( -1\right)
^{j}\left( 1-q\right) ^{n/2-j}q^{j\left( j+1\right) /2}\rho ^{n-i-2j}\QATOPD[
] {n-j}{j}_{q}\QATOPD[ ] {n-2j}{i}_{q}\allowbreak \times \allowbreak \\
&&\sum_{k=i}^{n-2j}\QATOPD[ ] {n-i-2j}{k-i}_{q}B_{k-i}\left( y|q\right)
H_{n-k-2j}\left( y|q\right) .
\end{eqnarray*}%
$\allowbreak $

Now $\sum_{k=i}^{n-2j}\QATOPD[ ] {n-i-2j}{k-i}_{q}B_{k-i}\left( y|q\right)
H_{n-k-2j}\left( y|q\right) \allowbreak =\allowbreak $\newline
$\sum_{s=0}^{n-i-2j}\QATOPD[ ] {n-i-2j}{s}_{q}B_{s}\left( y|q\right)
H_{n-i-2j-s}\left( y|q\right) \allowbreak =\allowbreak \left\{ 
\begin{array}{ccc}
1 & if & n-i=2j \\ 
0 & if & n-i>2j%
\end{array}%
\right. $ by Lemma \ref{znane} ii). \newline
Hence $G_{i,n}\left( y,\rho ,q\right) \allowbreak =\allowbreak \left\{ 
\begin{array}{ccc}
0 & if & n-i~~\text{is odd} \\ 
(-1)^{m}(1-q)^{n/2-m}q^{m(m+1)/2}\QATOPD[ ] {n-m}{m}_{q} & if & n-i=2m%
\end{array}%
\right. $. So $\sum_{k=0}^{n}D_{k,n}\left( y,\rho ,q\right) P_{k}\left(
x|y,\rho ,q\right) \allowbreak =\allowbreak \sum_{i=0}^{n}H_{i}\left(
x|q\right) G_{i,n}\left( y,\rho ,q\right) \allowbreak =\allowbreak
\sum_{m=0}^{\left\lfloor n/2\right\rfloor }(-1)^{m}(1-q)^{n/2-m}q^{m(m+1)/2}%
\QATOPD[ ] {n-m}{m}_{q}H_{n-2m}\left( x|q\right) \allowbreak =\allowbreak
U\left( x\sqrt{1-q}/2\right) $ Lemma \ref{znane} iv).

$ii)$ Notice that $C_{0,0}\left( y,\rho ,q\right) \allowbreak =\allowbreak
1, $ $C_{n,n}\left( y,\rho ,q\right) \allowbreak =\allowbreak (1-q)^{n/2},$ $%
C_{0,n}\left( y,\rho ,q\right) \allowbreak =\allowbreak (1-\rho
^{2})\allowbreak \times \allowbreak \sum_{j=1}^{\left\lfloor
n/2\right\rfloor }(-1)^{j}(1-q)^{n/2-j}q^{n+j(j-3)/2}\allowbreak \times
\allowbreak \QATOPD[ ] {n-1-j}{j}_{q}\rho ^{n-2j}H_{n-2j}\left( y|q\right) $%
, $C_{n-1,n}\left( y,\rho ,q\right) \allowbreak =\allowbreak \left(
1-q\right) ^{n/2}q\rho y\left[ n-1\right] _{q}.$ Hence $C_{0,1}(y,\rho
,q)\allowbreak =\allowbreak 0$ and $C_{1,1}(y,\rho ,q)\allowbreak
=\allowbreak (1-q)^{1/2},$ $C_{0,2}\left( y,\rho ,q\right) \allowbreak
=\allowbreak -(1-\rho ^{2}),$ $C_{1,2}\left( y,\rho ,q\right) \allowbreak
=\allowbreak (1-q)q\rho y.$ Thus equation (\ref{expansion}) is satisfied for 
$n\allowbreak =\allowbreak 0,1,2.$ For larger $n$ formula will be proved
straightforwardly. Let us consider an expression $W_{n}\left( x|y,\rho
,q\right) \allowbreak =\allowbreak \sum_{k=0}^{n}C_{k,n}\left( y,\rho
,q\right) P_{k}\left( x|y,\rho ,q\right) \allowbreak .$ We have%
\begin{eqnarray*}
W_{n}\left( x|y,\rho ,q\right) &=&\sum_{k=0}^{n}P_{k}\left( x|y,\rho
,q\right) \sum_{j=0}^{\left\lfloor (n-k)/2\right\rfloor }\left( -1\right)
^{j}\left( 1-q\right) ^{n/2-j}q^{n-k+j\left( j-3\right) /2}\QATOPD[ ] {n-1-j%
}{n-k-2j}_{q} \\
&&\times \left( \QATOPD[ ] {j+k}{k}_{q}-\rho ^{2}q^{k}\QATOPD[ ] {j+k-1}{k}%
_{q}\right) \rho ^{n-k-2j}H_{n-k-2j}\left( y|q\right) \\
&=&\sum_{j=0}^{\left\lfloor n/2\right\rfloor }\left( -1\right) ^{j}\left(
1-q\right) ^{n/2-j}q^{j(j+1)/2}\sum_{k=0}^{n-2j}\QATOPD[ ] {n-1-j}{n-k-2j}%
_{q} \\
&&\times \left( \QATOPD[ ] {j+k}{k}_{q}-\rho ^{2}q^{k}\QATOPD[ ] {j+k-1}{k}%
_{q}\right) \rho ^{n-k-2j}H_{n-k-2j}\left( y|q\right) P_{k}\left( x|y,\rho
,q\right)
\end{eqnarray*}

Now $n-k+j\left( j-3\right) /2\allowbreak =\allowbreak j\left( j+1\right)
/2\allowbreak +\allowbreak n-k-2j,$ $\QATOPD[ ] {n-1-j}{n-k-2j}_{q}\QATOPD[ ]
{j+k}{k}_{q}\allowbreak =\allowbreak \frac{\left[ n-1-j\right] _{q}!\left[
j+k\right] _{q}}{\left[ n-k-2j\right] _{q}!\left[ k\right] _{q}!\left[ j%
\right] _{q}!}\allowbreak =\allowbreak \frac{\lbrack j+k]_{q}}{\left[ n-j%
\right] _{q}}\QATOPD[ ] {n-j}{j}_{q}\QATOPD[ ] {n-2j}{k}_{q}$ and $\QATOPD[ ]
{n-1-j}{n-k-2j}_{q}\QATOPD[ ] {j+k-1}{k}_{q}\allowbreak =\allowbreak \frac{%
\left[ n-1-j\right] _{q}!}{\left[ n-k-2j\right] _{q}!\left[ k\right] _{q}!%
\left[ j-1\right] _{q}!}\allowbreak =$\allowbreak $\QATOPD[ ] {n-1-j}{j-1}%
_{q}\QATOPD[ ] {n-2j}{k}_{q},$ hence%
\begin{eqnarray*}
W_{n}\left( x|y,\rho ,q\right) &=&\sum_{j=0}^{\left\lfloor n/2\right\rfloor
}\left( -1\right) ^{j}\left( 1-q\right) ^{n/2-j}q^{j(j+1)/2}\frac{1}{\left[
n-j\right] _{q}}\QATOPD[ ] {n-j}{j}_{q} \\
&&\times \sum_{k=0}^{n-2j}\QATOPD[ ] {n-2j}{k}_{q}q^{n-k-2j}\rho
^{n-k-2j}H_{n-k-2j}\left( y|q\right) P_{k}\left( x|y,\rho ,q\right) \\
&&-\rho ^{2}\sum_{j=1}^{\left\lfloor n/2\right\rfloor }\left( -1\right)
^{j}\left( 1-q\right) ^{n/2-j}q^{j(j+1)/2}q^{n-2j}\QATOPD[ ] {n-1-j}{j-1}_{q}
\\
&&\times \sum_{k=0}^{n-2j}\QATOPD[ ] {n-2j}{k}_{q}\rho
^{n-k-2j}H_{n-k-2j}\left( y|q\right) P_{k}\left( x|y,\rho ,q\right) .
\end{eqnarray*}%
Now we apply Lemma \ref{znane} iii) and also the simple fact that $%
q^{n-k-2j}[k+j]_{q}\allowbreak =\allowbreak \lbrack n-j]_{q}\allowbreak
-\allowbreak \lbrack n-k-2j]_{q}.$ We get after applying Lemma \ref{znane}
iv)%
\begin{eqnarray*}
W_{n}\left( x|y,\rho ,q\right) &=&U_{n}\left( x\sqrt{1-q}/2\right)
-\sum_{j=0}^{\left\lfloor n/2\right\rfloor }\left( -1\right) ^{j}\left(
1-q\right) ^{n/2-j}q^{j(j+1)/2}\frac{[n-2j]_{q}}{\left[ n-j\right] _{q}}%
\QATOPD[ ] {n-j}{j}_{q} \\
&&\times \sum_{k=0}^{n-2j-1}\QATOPD[ ] {n-2j-1}{k}_{q}\rho
^{n-2j-k}H_{n-k-2j}\left( y|q\right) P_{k}\left( x|y,\rho ,q\right) \\
&&-\rho ^{2}\sum_{j=1}^{\left\lfloor n/2\right\rfloor }\left( -1\right)
^{j}\left( 1-q\right) ^{n/2-j}q^{j(j+1)/2}q^{n-2j}\QATOPD[ ] {n-1-j}{j-1}%
_{q}H_{n-2j}\left( x|q\right) .
\end{eqnarray*}%
Now we apply formula $H_{n-k-2j}\left( y|q\right) \allowbreak =\allowbreak
yH_{n-1-k-2j}(y|q)\allowbreak -\allowbreak \lbrack
n-1-2j-k]_{q}H_{n-2-2j-k}(y|q)$ and split the first sum into two. Since $%
\frac{[n-2j]_{q}}{\left[ n-j\right] _{q}}\QATOPD[ ] {n-j}{j}_{q}\allowbreak
=\allowbreak \QATOPD[ ] {n-1-j}{j}_{q}$ we see that the first of these two
sums is equal to $\rho \sqrt{1-q}yU_{n-1}\left( x\sqrt{1-q}/2\right) .$
Hence 
\begin{eqnarray*}
W_{n}\left( x|y,\rho ,q\right) &=&U_{n}\left( x\sqrt{1-q}/2\right) -\rho 
\sqrt{1-q}yU_{n-1}\left( x\sqrt{1-q}/2\right) \\
&&+\sum_{j=0}^{\left\lfloor n/2\right\rfloor }\left( -1\right) ^{j}\left(
1-q\right) ^{n/2-j}q^{j(j+1)/2}\frac{[n-2j]_{q}}{\left[ n-j\right] _{q}}%
\QATOPD[ ] {n-j}{j}_{q} \\
&&\times \sum_{k=0}^{n-2j-1}\QATOPD[ ] {n-2j-1}{k}_{q}[n-1-k-2j]_{q}\rho
^{n-2j-k}H_{n-2-k-2j}\left( y|q\right) P_{k}\left( x|y,\rho ,q\right) \\
&&+\rho ^{2}\sum_{j=0}^{\left\lfloor n/2-1\right\rfloor }\left( -1\right)
^{j}\left( 1-q\right) ^{n/2-1-j}q^{j(j+1)/2}q^{n-j-1}\QATOPD[ ] {n-2-j}{j}%
_{q}H_{n-2-2j}\left( x|q\right) .
\end{eqnarray*}%
Notice that 
\begin{equation*}
\sum_{k=0}^{n-2j-1}\QATOPD[ ] {n-2j-1}{k}_{q}[n-1-k-2j]_{q}\rho
^{n-2j-k}H_{n-2-k-2j}\left( y|q\right) P_{k}\left( x|y,\rho ,q\right)
\allowbreak =\allowbreak \lbrack n-1-2j]_{q}\rho ^{2}H_{n-2-2j}\left(
x|q\right)
\end{equation*}
by Lemma \ref{znane} iii). Besides $\frac{[n-2j]_{q}}{\left[ n-j\right] _{q}}%
\QATOPD[ ] {n-j}{j}_{q}\allowbreak =\allowbreak \QATOPD[ ] {n-1-j}{j}_{q}.$
Thus the sum of the last two summands is equal to\newline
\begin{eqnarray*}
&&\rho ^{2}(1-q)\sum_{j=0}^{\left\lfloor n/2\right\rfloor -1}\left(
-1\right) ^{j}\left( 1-q\right) ^{n/2-1-j}q^{j(j+1)/2}\QATOPD[ ] {n-1-j}{j}%
_{q}\left[ n-1-2j\right] _{q}H_{n-2-2j}\left( x|q\right) \allowbreak
+\allowbreak \\
&&\rho ^{2}\sum_{j=0}^{\left\lfloor n/2-1\right\rfloor }\left( -1\right)
^{j}\left( 1-q\right) ^{n/2-1-j}q^{j(j+1)/2}q^{n-j-1}\QATOPD[ ] {n-2-j}{j}%
_{q}H_{n-2-2j}\left( x|q\right) .
\end{eqnarray*}
Now 
\begin{equation*}
\QATOPD[ ] {n-1-j}{j}_{q}\left[ n-1-2j\right] _{q}\allowbreak =\allowbreak %
\left[ n-1-j\right] _{q}\QATOPD[ ] {n-2-2j}{j}_{q}
\end{equation*}
and $\left( 1-q\right) \left[ n-1-j\right] \allowbreak =\allowbreak
1-q^{n-1-j}$, hence the sum of last two summands is equal to \newline
\begin{equation*}
\rho ^{2}\sum_{j=0}^{\left\lfloor n/2-1\right\rfloor }\left( -1\right)
^{j}\left( 1-q\right) ^{n/2-1-j}q^{j(j+1)/2}\QATOPD[ ] {n-2-j}{j}%
_{q}H_{n-2-2j}\left( x|q\right) \allowbreak \allowbreak =\allowbreak \rho
^{2}U_{n-2}\left( x\sqrt{1-q}/2\right)
\end{equation*}
by Lemma \ref{znane} iv)
\end{proof}

\end{document}